\documentclass{amsart}
\usepackage{amsthm}
\usepackage{amssymb}
\usepackage{mathtools}
\usepackage{tikz-cd}
\usepackage{mathabx}
\usepackage{amsmath}
\usepackage[all]{xy}
\usepackage{undertilde}
\usepackage[enableskew]{youngtab} 
\theoremstyle{plain}
\newtheorem*{theorem*}{Theorem}
\newtheorem{theo}{Theorem}[section]
\newtheorem{cor}[theo]{Corollary}

\newtheorem{prop}[theo]{Proposition}
\newtheorem{lm}[theo]{Lemma}
\newtheorem{rem}[theo]{Remark}

\newtheorem*{lm*}{Lemma}
\newtheorem*{theo*}{Theorem}
\newtheorem*{cor*}{Corollary}
\newtheorem*{prop*}{Proposition}
\newtheorem*{def*}{Definition}
\newtheorem*{rem*}{Remark}

\begin{document}
	\title{Integral $Ext^2$ between hook Weyl modules}

	\author{Dimitra-Dionysia Stergiopoulou}
	\address{Department of Mathematics, University of Athens}
	\email{dstergiop@math.uoa.gr}

	
	
	\date{June 5, 2020}

	\maketitle
	\begin{abstract}
		This paper concerns representations of the integral general linear group. The extension groups $Ext^2$ between any pair of hook Weyl modules are determined via a detailed study of cyclic generators and relations associated to certain extensions. As a corollary, the modular extension groups $Ext^1$ between such modules are determined.
	\end{abstract}
	
	\section {Introduction}	In the study of polynomial representations of the integral general linear group $GL_n$, the Weyl modules $\Delta(\lambda)$, indexed by partitions $\lambda$, play a central role.
	These modules have an explicit construction, enjoy a standard basis theorem and their characters are the classical Schur functions \cite{Gr}. Moreover, the heads of the modules $K\otimes \Delta(\lambda)$ form a complete set of inequivalent polynomial simple modules for $GL_n(K)$, where $K$ is an infinite field.  However, the structure of the modules $\Delta(\lambda)$ is not well understood.
	
	One of the important problems in the area is to determine the extension groups $Ext^i(\Delta(\lambda),\Delta(\mu))$. Relatively few explicit results are known. Modular extension groups for $GL_n(K)$ were studied in \cite{A}, when $\lambda$ consists of a single column and $\mu$ of a single row. For $SL_2(K)$, all modular extension groups between Weyl modules were described in \cite{Pa}, generalizing \cite{Er} and \cite{CE}. The modular extension groups $Ext^1(K\otimes\Delta(\lambda), K\otimes\Delta(\mu))$ for $GL_n(K)$ were determined in \cite{DG}, when $\mu$ consists of a single row.  
	
	Concerning the integral extension groups, $Ext^1(\Delta(\lambda),\Delta(\mu))$ for $GL_3$ was computed in \cite{BF}, when $\lambda$ and $\mu$ differ by a multiple of a positive root. The groups $Ext^1(\Delta(\lambda),\Delta(\mu))$ for $GL_n$ was determined in \cite{Ku2}, when $\lambda$ and $\mu$ differ by a single positive root. The groups $Ext^i(\Delta(h), \Delta (h(k)))$ were computed in \cite{MS} for $i=1,k$ and any pair of hooks $h =(a,1^b)$ and $h(k)=(a+k, 1^{b-k})$. The main method there was to determine cyclic generators of certain extension groups of the form $Ext^i(\Delta(h),D_{a+k}\otimes\Lambda^{b-k})$, where $D_*$ and $\Lambda^*$ denote the divided power algebra and the exterior algebra, respectively, of the natural $GL_n$-module. Except for some extreme cases, these generators usually have an involved form. The next step in our method was to calculate the images of these under canonical maps. The idea of using the `skew' extensions $Ext^i(\Delta(h),D_{a+k}\otimes\Lambda^{b-k})$, in place of the usual extensions $Ext^i(\Delta(h),\Delta(h(k)))$, comes from the observation that the former seem to have more manageable presentation matrices. The reason for this is that there is no straightening law involved.
	
	The purpose of the present paper is to study $Ext^2$ between any two hook Weyl modules using the above strategy. The main result is the following.
	
	\begin{theo} Consider hook partitions $h=(a,1^b)$ and $h(k)=(a-k,1^{b-k})$ for $k \ge 0$. Suppose $n\ge b+1$. Then, $Ext^2(\Delta(h), \Delta(h(k)))=0$, except possibly when $k\in\{2,3,4\},$ in which cases
		$$Ext^2(\Delta(h), \Delta(h(k))) = \begin{cases}  \mathbb{Z}_s, & \mbox {if}\;k=2, 
		\\ \mathbb{Z}_{3/t}, & \mbox {if}\;k=3,
		\\ \mathbb{Z}_{t}, & \mbox {if}\;k=4,
		\end{cases} $$
		where $s= {(a+b)}/{gcd(2,a+b)}$, $t={gcd(3,a+b)}$.
	\end{theo}
As a corollary, we determine the modular $Ext^1$ groups between any hook partitions.

In Section 2 we gather necessary preliminaries. In Section 3 we consider the cases $k\neq 4$. In Section 4, which costitutes the main part of the paper, we complete the proof of the theorem. As a corollary, we determine the dimensions of the modular $Ext^1$ groups between any hook Weyl modules. 
	
	\section{Preliminaries}
	We will use the notation of \cite{MS} but, for completeness, we recall the main points here and supplement with additional material. 
	\subsection{Notation.} 
	Let $F$ be a free abelian group of finite rank $n$.  Fixing a basis of $F$ yields an identification of general linear groups $GL(F)=GL_n(\mathbb{Z})$. We will be working with homogeneous polynomial representations of $GL_n(\mathbb{Z})$ of degree $r$, or equivalently, with modules over the Schur algebra $S_\mathbb{Z}(n,r)$  \cite{Gr}. We will write $S(n,r)$ in place of $S_\mathbb{Z}(n,r)$.
	By $DF=\sum_{i\geq 0}D_iF$ and $\Lambda F=\sum_{i\geq 0}\Lambda^{i}F$
	we denote the divided power algebra of $F$ and the exterior algebra of $F$ respectively. We will usually omit $F$ and write $D_i$ and $\Lambda^i$.
	
 Throughout this paper all tensor products are over the integers.
	
	For a partition $\lambda$ of $r$ with at most $n$ parts, we denote by $\Delta(\lambda)$ the corresponding Weyl module for $S(n,r)$. A hook $h$ is a partition of the form $h=(a,1^b)$. 
	
	The following complex of $S(n,r)$-modules (which is the dual of the usual Koszul complex) is exact 
	$$0 \rightarrow D_{a+b} \rightarrow ... \rightarrow D_{a+1} \otimes \Lambda^{b-1} \xrightarrow{\Theta_{a}}  D_{a} \otimes \Lambda^{b} \rightarrow ... \rightarrow \Lambda^{a+b} \rightarrow 0,$$
	where $\Theta_{a}$ is the composition 
	$\ D_{a+1}\otimes \Lambda^{b-1} \xrightarrow{\triangle \otimes 1} D_{a}\otimes D_1 \otimes \Lambda^{b-1}\xrightarrow{1 \otimes m} D_{a} \otimes \Lambda^{b},$
	where $\triangle$ (respectively, $m$) is the indicated component of the comultiplication (resp., multiplication) map of the Hopf algebra $DF$ (resp., $\Lambda F$). It is  well known that if $h=(a,1^b)$ is a hook, $b \geq 1$, then
	$\Delta (h) \simeq cok(\Theta_{a}) \simeq ker(\Theta_{a-1}).$
	
	Throughout this paper we use the notation $h=(a,1^b),h(k)=(a+k,1^{b-k})$ and $r=a+b.$
	
    \subsection{Straightening law and standard basis.} We recall the straightening law and the standard basis theorem for $\Delta(h)$ \cite{ABW}. Fix an ordered basis $e_1,\dots,e_n$ of $F$. For simplicity, we denote the element $e_i$ by $i$ and accordingly the element $e_{i_1}^{(a_1)} \cdots  e_{i_t}^{(a_t)} \otimes e_{j_1} \wedge \dots \wedge e_{j_b} \in D_a \otimes \Lambda^b$ by ${i_1}^{(a_1)}  \cdots  {i_t}^{(a_t)} \otimes {j_1} \cdots {j_b}$. The image of this element under the identification $\Delta (h) \simeq cok(\Theta_{a})$ will be denoted by ${i_1}^{(a_1)}  \cdots  {i_t}^{(a_t)} / {j_1} \cdots {j_b}$. Now suppose $i_1<i_2< \dots <i_t$ and $j_1 \le i_1$. Then in $\Delta(h)$ we have 
	$${i_1}^{(a_1)}\cdots {i_t}^{(a_t)} / {j_1} \cdots {j_b} = \begin{cases} -\sum\limits_{s\geq 2}{i_1}^{(a_1+1)}  \cdots {i_s}^{(a_s-1)}\cdots{i_t}^{(a_t)} / {i_s}{j_2} \cdots {j_b}, & \mbox{if}\;j_1=i_1 \\ -\sum\limits_{s\geq 1}j_1{i_1}^{(a_1)}  \cdots {i_s}^{(a_s-1)}\cdots {i_t}^{(a_t)} / {i_s}{j_2} \cdots {j_b}, & \mbox {if}\;j_1<i_1. \end{cases} $$
	
	\noindent A $\mathbb{Z}$ - basis of $\Delta(h)$ is the set of elements ${i_1}^{(a_1)}  \cdots {i_t}^{(a_t)} / {j_1} \cdots {j_b} $, where $a_1+\cdots+a_t=a$, $i_1<\dots<i_t$ and $i_1 <j_1< \dots <j_b.$
	\subsection{Resolutions of hooks.}	From \cite{Gr} or \cite{AB} we recall that for each sequence $a_1,\dots,a_n$ of non negative integers $a_i$ that sum to $r$, the $S(n,r)$-module $D_{a_1} \otimes \cdots \otimes D_{a_n}$ is projective. 
	
	We will use the explicit finite projective resolution $P_{*}(a,b)$ of $ \Delta (h)$,
	\[0 \longrightarrow \cdots \longrightarrow P_2(a,b) \xrightarrow{\Theta_{2}(a,b)} P_1(a,b) \xrightarrow{\Theta_{1}(a,b)} P_0(a,b)
	\]
	of \cite{Ma} which we now recall. For short, we denote the tensor product $D_{a_1} \otimes \cdots \otimes D_{a_m}$ of divided powers by $D(a_1,\dots,a_m)$. We have
	$P_i(a,b)=\sum D(a_1,\dots,a_{b+1-i})
	$  
	where the sum ranges over all sequences $(a_1,\dots,a_{b+1-i})$ of positive integers of length $b+1-i$ such that $a_1+\cdots+a_{b+1-i}=a+b$ and $a \le a_1 \le a+i$. The differential $\Theta_{i}(a,b)$ is defined be sending $x_1 \otimes \cdots \otimes x_{b+1-i} \in D(a_1,\dots,a_{b+1-i})$ to
	\[ \sum_{j=1}^{b+1-i} (-1)^{j+1}x_1 \otimes \cdots \otimes \triangle(x_j) \otimes \cdots  \otimes x_{b+1-i} \in D(a_1,\dots,u,v,\dots,a_{b+1-i}),
	\]
	where  $\triangle$ is the the two-fold diagonalization $D(a_j) \rightarrow \sum D(u,v)$, the sum ranges of all positive integers $u,v$ such that $u+v=a_j$ and $D(a_1,\dots,u,v,\dots,a_{b+1-i})$ is a summand of $P_{i-1}(a,b)$ with $u$ located in position $j$. Let $ \triangle_{u,v} : D(a_j) \rightarrow D(u,v)$ be the indicated component of the two-fold diagonalization $D(a_j) \rightarrow \sum D(u,v)$.

If $A,B$ are $S(n,r)$ - modules, we write $Hom(A,B)$ and $Ext^i(A,B)$ in place of $Hom_{S(n,r)}(A,B)$ and $Ext^i_{S(n,r)}(A,B)$ respectively.

We recall the recursions  \begin{align*}
&P_0(a,b)=D(a) \otimes P_0(1,b-1), \\ &P_i(a,b)=P_{i-1}(a+1,b-1) \oplus D(a) \otimes P_i(1,b-1), i>0 \end{align*}
and that under these identifications we have the following.
	\begin{rem} The differential $Hom(\Theta_i(a,b), M)$ of the complex \\ $Hom(P_*(a,b),M)$ is given by
		\begin{center}
		\begin{tikzcd}
			Hom(P_{i-2}(a+1,b-1),M) \arrow{r}  \arrow[d, phantom, "\oplus"]
			& Hom(P_{i-1}(a+1,b-1),M)\arrow[d, phantom, "\oplus"]\\
			Hom(D(a) \otimes P_{i-1}(1,b-1),M)\arrow{r} \urar[shorten >= 25pt,shorten <= 25pt]{}
			&Hom(D(a) \otimes P_i(1,b-1),M)
		\end{tikzcd}
	\end{center}  
		where the top horizontal map is $Hom(\Theta_{i-1}(a+1,b-1),M)$, the bottom one is $-Hom(1 \otimes  \Theta _i(1,b-1),M)$ and the restriction of the diagonal one on the summand $Hom(D(a,j,a_2,\dots,a_m),M)$  is $Hom(\triangle_{a,j} \otimes 1 \otimes \cdots \otimes 1), M) .$
	\end{rem}
We now consider weight spaces. For any $S(n,r)$-module $M$ and any sequence $a_1,\dots,a_m$ of nonnegative integers such that $a_1+\cdots +a_m=r$ and $m \le n$,  we identify the $\mathbb{Z}$-module $Hom(D(a_1,\dots,a_m),M)$ with the $(a_1,\dots,a_m)$ weight subspace of $M$ (with respect to the action of $\mathbb{Z}^n$) \cite{AB}. We will use such identifications freely throughout this paper.
	In particular, suppose  $M$ is a skew Weyl module for $S(n,r)$ (denoted be $K_{\lambda / \mu}$ in \cite{ABW}). Using the $\mathbb{Z}$-basis of $M$  given by standard tableaux of \cite{ABW}, Thm. II.3.16, we see that the $\mathbb{Z}$-module $Hom(D(a_1,\dots,a_m),M)$ may be identified with the $\mathbb{Z}$-submodule of $M$ that has basis the standard tableaux of $M$ that contain the entry $i$ exactly $a_i$ times, $i\in\{1,\dots,m\}$. We call this the standard basis of $Hom(D(a_1,\dots,a_m),M)$. If $a$ is a non negative integer, we denote by $M_a$ the $\mathbb{Z}$-submodule of $M$ given by 
	\[M_a=\sum Hom(D(a,a_2,\dots,a_m),M)\]
	where the sum ranges over all nonnegative integers $a_2,\dots,a_m$ such that $a_2+\cdots+a_m$ $=r-a$ and $m \le n$. A  $\mathbb{Z}$- basis of this is the set of standard tableaux of $M$ in which the entry 1 appears exactly $a$ times. $M_a$ is a $GL_{n-1}$ submodule of $M$ with $GL_{n-1}$ acting on the basis elements $2,\dots,n$. 
	
From the definition of $M_a$ it follows that we have the identification $$Hom_A(D(a_1,\dots,a_m),M)=Hom_B(D(a_2,\dots,a_m),M_{a_1})$$ and thus by summing,  $Hom_A(D(a) \otimes P_i(1,b-1),M)=Hom_B(P_i(1,b-1),M_{a}),$ where $a=a_1$, $A=S(n,r), B=S(n-1,b)$. We refer to the standard basis of $Hom_A(D(a_1,\dots,a_m),M)$ as the standard basis of $Hom_B(D(a_2,\dots,a_m),M_a)$. 
	
Under the above identifications it follows from the definition of the complexes, that the maps 
	\begin{align*}
	&Hom(1 \otimes  \Theta _i(1,b-1),M): \\
	&Hom_A(D(a) \otimes P_{i-1}(1,b-1),M) \rightarrow Hom_A(D(a) \otimes P_i(1,b-1),M), \\
	&Hom(\Theta _i(1,b-1),M_a): \\
	&Hom_B(P_{i-1}(1,b-1),M_a) \rightarrow Hom_B(P_i(1,b-1),M_a) \end{align*} 
	are equal. From now on we will drop from the above $Hom$ modules the subscripts $A,B$. 
	\subsection{Additional Preliminaries.}
	We want to consider matrices of the differential of the complex $Hom(P_*(a,b),M)$, where $M$ is a skew Weyl module.
	
	We consider the usual lexicographic ordering of the elements of the standard basis of $Hom(D(a_1,\dots,a_m),M)$ $=Hom(D(a_2,\dots,a_m),M_{a_1})$. We identify the sequence $(a_1,\dots,a_m)$ of nonnegative integers, where $m \le n$,  with the sequence  $(a_1,\dots,a_m,0,\dots,0) \in \mathbb{Z}^n $. Now if  $(a_1,\dots,a_m)$ is greater than   $(b_1,\dots,b_{m'})$ , where $m,m' \le n$, in the usual lexicographic ordering of sequences, me declare that each element of the standard basis of $Hom(D(a_1,\dots,a_m),M)$ is less than each element of the standard basis of $Hom(D(b_1,\dots,b_{m'}),M)$.

	For $M$ a skew Weyl module, let $E^{i}(\Delta(h),M)$ be the cokernel of the differential $Hom(\Theta_i (a,b),M)$ of the complex $Hom(P_{*}(a,b),M)$. We know that the torsion part of this abelian group is isomorphic to $Ext^{i}(\Delta(h),M)$ \cite{AB}. %So, let $\pi: Hom(P_i (a,b),M)\rightarrow E^{i}(\Delta(h),M)$ the natural projection. Throughout this paper we use to identify the elements of $E^{i}(\Delta(h),M)$ with their preimages under $\pi$ , i.e. we write $T$ instead of $\pi (T)$ for $T\in Hom(\Theta_i (a,b),M) $.%

Let $e^{(i)}(a,b,M)$ be the matrix of the differential $Hom(\Theta_i(a,b),M)$ with respect to the standard bases for the various spaces $Hom(D(a_1,\dots,a_m),M)$ and the above ordering. Likewise, let $e^{(i)}(1,b-1,M_a)$ be the matrix of the differential $Hom(\Theta_i(1,b-1),M_a)$ with respect to the standard bases for the various $Hom(D(a_2,\dots,a_m),M_a)$. From the previous discussion and Remark 2.1 we get the following statement.
	\begin{lm} \label{canonicalmatrix}
		Suppose $n \ge b+1$ and $M$ is a skew Weyl module for $S(n,r)$. Then for $i=1$
		\[
		e^{(1)}(a,b,M)=
		\left(
		\begin{array}{c}
		B^1(a,b,M)\\
		\hline \;
		-e^{(1)}(1,b-1,M_a)
		\end{array}
		\right)
		\]
		and for $i>1$ 
		\[
		e^{(i)}(a,b,M)=
		\left(
		\begin{array}{c|c}
		e^{(i-1)}(a+1,b-1,M) &B^i(a,b,M)\\
		\hline \;
		& -e^{(i)}(1,b-1,M_a)
		\end{array}
		\right),
		\]
		where in both cases $B^i(a,b,M)$ is the matrix of  the diagonal map $Hom(D(a) \otimes P_{i-1}(1,b-1),M) \rightarrow Hom(P_{i-1}(a+1,b-1),M)$ of Remark 2.1.
	\end{lm}
	We will also need a lemma for the special case $M=D_{a+k}\otimes\Lambda^{b-k}$ which we now describe. We will need a different order of the basis elements.
	
	Let $B_i$ be the standard basis of $Hom(P_{i-1}(a+1,b-1,D_{a+k}\otimes\Lambda^{b-k}))$, $B_{i,1}$ the subset of the standard basis of $Hom(D(a) \otimes P_{i}(1,b-1),D_{a+k}\otimes\Lambda^{b-k}))$ consisting of the standard tableaux with one 1 in the $\Lambda^{b-d}$ part and $B_{i,0}$ the subset of the standard basis of $Hom(D(a) \otimes P_{i}(1,b-1),D_{a+k}\otimes\Lambda^{b-k}))$ consisting of the standard tableaux with no 1 in the $\Lambda^{b-k}$ part. We order each of the sets $B_i, B_{i,1}, B_{i,0}$ lexicographically and we declare that the elements of $B_i$ are less than the elements of $B_{i,1}$, and the elements of $B_{i,1}$ are less than the elements of $B_{i,0.}$ Hence we have a total order on the basis $B_{i-1} \cup B_{i-1,1} \cup B_{i-1,0}$ of the domain of $Hom(\Theta_i(a,b),D_{a+d}\otimes\Lambda^{b-k})$ and a total order on the basis $B_{i} \cup B_{i,1} \cup B_{i,0}$ of the codomain of $Hom(\Theta_i(a,b),D_{a+k}\otimes\Lambda^{b-k})$.
	\begin{lm} \label{skewmatrix}
		With respect to the above ordered bases and for $i \ge2$,
		\[e^{(i)}(\Delta(h), D_{a+k}\otimes\Lambda^{b-k})=
		\left(
		\begin{array}{c|c|c}
		A &*&*\\
		\hline \;
		
		&B& \\
		\hline
		&&C
		\end{array}
		\right) ,
		\]
		where \begin{align*}
		A &=e^{(i-1)}(\Delta(a+1,1^{b-1})
		,D_{a+k}\otimes\Lambda^{b-k}),\\ B&=-e^{(i)}(\Delta(1,1^{b-1}),D_{1+k}\otimes\Lambda^{b-k-1}),\\ C&=-e^{(i)}(\Delta(1,1^{b-1}),D_{k}\otimes\Lambda^{b-k}). \end{align*}
	\end{lm}
	\begin{proof}The claim for the matrix $A$ follows from Lemma 2.2. For $B$ consider the diagram of abelian groups
	\begin{center}
	\begin{tikzcd}
		spanB_{i-1,1} \arrow{r}{a} \arrow[d, "f_{i-1}"]
	&spanB_{i,1}\arrow[d, "f_i"]\\
Hom(P_{i-1}(1,b-1),D_{1+k}\otimes\Lambda^{b-k-1})	\arrow{r}{\beta}
	&Hom(P_{i}(1,b-1),D_{1+k}\otimes\Lambda^{b-k-1})
	\end{tikzcd}
\end{center} 
where we regard the modules in the bottom row as $GL_{n-1}$ modules  with $GL_{n-1}$ acting on the basis elements $1,\dots,n-1$. Here,  $f_{i-1}$ is the isomorphism of abelian groups sending each standard basis element $1^{(a-1)}i_2^{(c_2)}\cdots i_q^{(c_q)}\otimes 1j_2\cdots j_{b-k} \in B_{i-1,1}$ to the standard basis element $(i_2-1)^{(c_2)}\cdots (i_q-1)^{(c_q)}\otimes (j_2-1)\cdots (j_{b-k}-1) \in Hom(P_{i-1}(1,b-1),D_{1+k}\otimes\Lambda^{b-k-1})$ and similarly for $f_i$, $\alpha$ is the restriction of $Hom(\Theta_i(a,b),D_{a+k}\otimes\Lambda^{b-k})$ to $spanB_{i-1,1}$ and $\beta = Hom(\Theta_i(1,b-1),D_{1+k}\otimes\Lambda^{b-k-1})$.We claim that the diagram commutes. Using Remark 2.1 and Remark 2.2 it suffices to show that $f_i\circ \theta_s = \theta_{s-1} \circ f_{i-1}$, $t \ge 2$. Indeed, if\[ T=1^{(a-1)}i_2^{(c_2)}\cdots i_q^{(c_q)}\otimes 1j_2\cdots j_{b-k} \in B_{i-1,1},
	\]then \[ \theta_s(T) = 1^{(a-1)}i_2^{(c_2)}\cdots i_u^{(c_u)}{i^{'}_{u+1}}^{(c_{u+1})}\cdots {i^{'}_{q}}^{(c_{q})}\otimes 1j_2\cdots j_vj^{'}_{v+1}\cdots j^{'}_{b-k},
	\]
	where $i_u=max(\{i_2,\dots,i_q\} \cap (\{2,\dots,s\})$, $j_v=max(\{j_2,\dots,i_{b-k}\} \cap (\{2,\dots,s\})$, 
	and $i^{'}=i-1$. We note that $\theta_s(T)$ is either 0 (if $j_v=j^{'}_{v+1}$) or a standard basis element (if $j_v \neq j^{'}_{v+1}$ and $i_u \neq i^{'}_{u+1}$ ) or a non-zero multiple of a standard basis element (if $j_v \neq j^{'}_{v+1}$ and $i_u = i^{'}_{u+1}$). Hence, in all cases we have
	\begin{align*}
	 f_i \circ \theta_s(T) =&  (i_2-1)^{(c_2)}\cdots(i_u-1)^{(c_u)}({i^{'}_{u+1}-1})^{(c_{u+1})}\cdots({i^{'}_{q}-1})^{(c_{q})}\otimes (j_2-1)\\
	 &\cdots(j_v-1)(j^{'}_{v+1}-1)\cdots(j^{'}_{b-k}-1).
	\end{align*}
	Similarly, one verifies that 
	\begin{align*}
	\theta_{s-1} \circ f_{i-1}  (T) =&  (i_2-1)^{(c_2)}\cdots(i_u-1)^{(c_u)}({i^{'}_{u+1}-1})^{(c_{u+1})}\cdots({i^{'}_{q}-1})^{(c_{q})}\otimes (j_2-1)\\
    &\cdots(j_v-1)(j^{'}_{v+1}-1)\cdots(j^{'}_{b-k}-1)
	\end{align*}
and thus the diagram commutes. The proof for $C$ is similar and thus omitted. Finally, the middle block in the right column block is indeed 0, since for every $T \in B_{i-1,0}$, and every $s \ge 2$ we have $\theta_s(T) $ is of the form $x\otimes j_1\dots j_{b-k}$ with $1 \notin \{ j_1,\dots,j_{b-k} \}$. \end{proof}
\noindent \textit{Remark.} We will apply many times the isomorphism $f_2$ of the previous proof in  section 4 in order to describe specific relations of $E^{2}(\Delta(h),D_{a+k}\otimes \Lambda^{b-k})$.

The next lemma will be used several times. The first equality follows from the main result of \cite{Ku1}.
	\begin{lm} \label{basiclemma}
		Suppose $n \ge {b+1}$ and $0 \le k < b$. Then $$Ext^i(\Delta (h),D_{a+k} \otimes \Lambda ^{b-k})=Ext^i(\Lambda ^{k+1}, D_{k+1}).$$ In particular,
		$\;$ \begin{enumerate}
		\item[{\rm (a)}] $Ext^1(\Delta (h),D_{a+k} \otimes \Lambda ^{b-k})= \mathbb{Z}_2, \;  k \ge 1,$  
		\item[{\rm (b)}] $Ext^2(\Delta (h),D_{a+k} \otimes \Lambda ^{b-k})= \begin{cases}\mathbb{Z}_3, \;  k=2,3   \\ 0, \; \;\; k \ne 2,3.\end{cases}$ \end{enumerate}
	\end{lm}
	\begin{proof} The first statement and (1) are Lemma 2.3 of \cite{MS}. In \cite{To}, Ex. 11.9, $Ext^2(\Lambda ^{k+1}, D_{k+1})$ was determined and hence we have (2). We give below a different proof of Touz\'e's result for the sake of self completeness. 
	
We calculate $Ext^2(\Lambda ^{k+1}, D_{k+1})$ using the projective resolution $P_*(1,k)$ of $\Lambda ^{k+1}$ and the method of \cite{AB}, Section 9. If $k=0,1$, the length $P_*(1,k)$ is less than 2 and hence $Ext^2(\Lambda ^{k+1}, D_{k+1})=0$. Suppose $k \ge 2$. According to Remark 2.1, the map $Hom( \Theta_2(1,k),D_{k+1})$ looks like:
	\begin{center}
	\begin{tikzcd}
	Hom(P_0(2,k-1),D_{k+1}) \arrow{r}  \arrow[d, phantom, "\oplus"]
	& Hom(P_1(2,k-1),D_{k+1})\arrow[d, phantom, "\oplus"]\\
	Hom(D(1) \otimes P_1(1,k-1),D_{k+1})\arrow{r} \urar[shorten >= 25pt,shorten <= 25pt]{}
	&Hom(D(1) \otimes P_2(1,k-1),D_{k+1})
	\end{tikzcd}
\end{center} 
	If $a_1,\dots,a_m$ are positive integers such that $a_1+\cdots+a_m=k+1,$ then we have $Hom(D(a_1,\dots,a_m),D_{k+1})= \mathbb{Z}$, 
	a generator is simply the multiplication map, which will be denoted by $1^{(a_1)}\cdots m^{(a_m)}$. Thus,  the map $$Hom(D(a_1,\dots,u,v,\dots,a_m),D_{k+1}) \rightarrow Hom(D(a_1,\dots,a_t,\dots,a_m),D_{k+1}),$$ where $u+v=a_t$, induced by $1 \otimes \cdots \otimes \triangle_{u,v} \otimes \cdots \otimes 1: D(a_1,\dots,a_t,\dots,a_m) \rightarrow D(a_1,\dots,u,v,\dots,a_m)$ is multiplication by the binomial coefficient $ \binom{a_t}{u}$.

	We have the following ordered $\mathbb{Z}$ - bases. 
	\begin{itemize}\item $Hom(P_0(2,k-1),D_{k+1}) : 1^{(2)}2\cdots k,$
	\item $Hom(D(1) \otimes P_1(1,k-1),D_{k+1}) :12^{(2)}3\cdots k, 123^{(2)}\cdots k, \dots,12\cdots k^{(2)},$
	\item $Hom(P_1(2,k-1),D_{k+1}) : 1^{(3)}2 \cdots (k-1), 1^{(2)}2^{(2)}4 \cdots (k-1),$\\
	$1^{(2)}23^{(2)}4 \cdots (k-1),\dots ,1^{(2)}23\cdots (k-1)^{(2)}.$
	\item $Hom(D(1) \otimes P_2(1,k-1),D_{k+1}) : 12^{(3)}3 \cdots (k-1), 12^{(2)}3^{(2)}4 \cdots (k-1),$\\
	$\dots, 12^{(2)}\cdots (k-1)^{(2)}, \dots ,123^{(3)}\cdots (k-1), 123^{(2)}4^{(2)}\cdots (k-1),$\\
	$\dots,123^{(2)}\cdots (k-1)^{(2)}, \dots , 12\cdots (k-1)^{(3)}.$
	\end{itemize}
	Now with the above ordered bases, a quick computation shows that the matrix of the top horizontal arrow of $Hom( \Theta_2(1,k),D_{k+1})$ is the $(k-1) \times 1$ matrix $$e^{(1)}(2,k-1,D_{k+1}) = (3 \; \; -2 \; \; \; 2 \; \; \; \dots \; \; (-1)^{k-2}2)^t$$ and the matrix of the diagonal arrow is the $(k-1) \times (k-1)$ matrix $$B^2(1,k,D_{k+1})=diag(3,2,\dots,2).$$ Also, it follows that the matrix  $e^{(2)}(1,k,D_{k+1})$ is of size $ \binom{k}{2} \times k$. In the notation of Lemma 2.2, we have $M=D_{k+1}$ and thus $M_1=D_k$. Therefore the recursion for $e^{(2)}(1,k,D_{k+1})$ is
	\[
	e^{(2)}(1,k,D_{k+1})=
	\left(
	\begin{array}{c|c}
	e^{(1)}(2,k-1,D_{k+1})  &B^2(1,k,D_{k+1})\\
	\hline \;
	& -e^{(2)}(1,k-1,D_k)
	\end{array}
	\right)
	\]
	where $k \ge 2$ and $e^{(2)}(1,2,D_3)=(3 \; \; 3)$ (this follows by an immediate computation). Hence, \[e^{(2)}(1,3,D_4)=  \begin{pmatrix}
	3 &3 & \\
	-2 & & 2 \\
	&-3 &-3
	\end{pmatrix}.\] The nonzero invariant factors of these last two matrices are 3 and 1,3 respectively and therefore $Ext^2( \Lambda ^3,D_3)=Ext^2( \Lambda ^4,D_4)= \mathbb{Z}_3$.
	
	Let $k \ge 4$. In order to show that $Ext^2(\Lambda ^{k+1}, D_{k+1})=0$, it suffices to show that the last nonzero invariant factor of $e^{(2)}(1,k,D_{k+1})$ is equal to 1. This matrix has rank at most $k-1$ since the matrix product $e^{(2)}(1,k,D_{k+1}) \cdot e^{(1)}(1,k,D_{k+1}) $ is zero, due to the fact that $Hom(P_*(1,k),D_{k+1})$ is a complex. (We have $e^{(1)}(1,k,D_{k+1}) \ne 0 $ since $Ext^1(\Lambda^{k+1},D_{k+1}) \ne 0$. In fact here  $e^{(1)}(1,k,D_{k+1})$ is the $k \times 1$ matrix $(2 \; \; -2 \; \; \; 2 \; \; \; \cdots \; \; (-1)^{k-1}2)^t)$.  It is straightforward to verify that the minor of $e^{(2)}(1,k,D_{k+1})$ corresponding to columns $2,3,\dots,k$ and rows $2,3,\dots,k-1, k+1$ is equal to $\pm 2^{k-1}$ and an immediate induction shows that the minor corresponding to columns $1,2,\dots,k-1$ and rows $1, 1+(k-1), 1+(k-1)+(k-2),\dots, \binom{k}{2}$ is equal to $\pm 3^{k-1}$. Since there exist two relatively prime minors of size the rank of the matrix $e^{(2)}(1,k,D_{k+1})$, the last nonzero invariant factor of this matrix is equal to 1. 
	\end{proof}

	Next we describe the differential of $Hom(\Theta_i (a,b),M)$ for $M$ a skew Weyl module. For $T \in Hom(D(a_1,\dots,a_m),M)$ a standard basis element, let $\theta_s (T)$, $1 \le s \le m$, be the element of $Hom(D(a_1,\dots,a_s+a_{s+1},\dots,a_m),M)$ obtained from $T$ by replacing each occurrence of $j>s$ by $j-1$. If $s>m$, let $\theta_s (T)=0$. By extending linearly, we obtain for each degree $i$ a map of $\mathbb{Z}$-modules $Hom(P_i(a,b),M) \rightarrow Hom(P_{i+1}(a,b),M)$ which is denoted by $\Theta_i$. It is clear that only a finite number of these maps are nonzero. From the definition of the differential of $P_*(a,b)$,	with the previous notation, we obtain the following fact.
	\begin{rem}
	 $Hom(\Theta_i(a,b),M) = \sum_{s \ge 1}(-1)^{s-1}\theta_s.$
	\end{rem}

    \section{The cases $k=3$ and $k\ge 5$}
	It is well known, for example by \cite{Ja}, B.3 (4), and \cite{Ma} respectively, that  $Ext^2(\Delta(h),\Delta(h))=$ $Ext^2(\Delta(h),\Delta(h(1)))=0$. Theorem 4.1  of \cite{MS} implies that $Ext^2(\Delta(h),h(2)))=\mathbb{Z}_{(a+b)/gcd(2,a+b)}$. 
	
	In this Section we will determine the extension groups $Ext^2(\Delta(h),\Delta(h(k)))$ for $k=3$ and $k\geq 5$. We will use certain exact sequences that we now describe.

Consider the short exact sequence
	\begin{equation} \label{s1}
	0\rightarrow\Delta(h(k+1))\xrightarrow{i_k} D_{a+k}\otimes \Lambda^{a-k}\xrightarrow{\pi_k} \Delta(h(k))\rightarrow 0
	\end{equation}
	for every $k\in\mathbb{Z}_{>0}$, where $\pi_k$ is induced by the identity map on generators and $i_k$ is induced by the composition $$D_{a+k+1}\otimes \Lambda^{b-1} \to D_{a+k}\otimes D_1 \otimes \Lambda^{b-1} \to D_{a+k}\otimes \Lambda^{b}$$ of comultiplication in $ D_* $ and multiplication in  $ \Lambda^*$.
	
	We have $Hom_{S_{\mathbb{Q}}(n,r)}(\Delta(h),\Delta(h(k)))=0$, since $\Delta(h)$ and $\Delta(h(k))$ are distinct irreducible representations of $S_{\mathbb{Q}}(n,r)$. Thus, $Hom(\Delta(h),\Delta(h(k)))=0$. Applying then $Hom(\Delta(h),-)$ to (\ref{s1}) we obtain the exact sequence
	\begin{align}\label{long1}
	0&\longrightarrow Ext^1(\Delta(h),\Delta(h(k+1)))\xrightarrow{i_k^{(1)}}
	Ext^1(\Delta(h),D_{a+k}\otimes\Lambda^{b-k})\nonumber \\
	&\xrightarrow{\pi_k^{(1)}}Ext^1(\Delta(h),\Delta(h(k))) \longrightarrow Ext^2(\Delta(h),\Delta(h(k+1)))\nonumber\\
	&\xrightarrow{i_k^{(2)}}Ext^2(\Delta(h),D_{a+k}\otimes\Lambda^{b-k})
	\xrightarrow{\pi_k^{(2)}} Ext^2(\Delta(h),\Delta(h(k))) \nonumber \\
	&\longrightarrow Ext^3(\Delta(h),\Delta(h(k+1)))\xrightarrow{i_k^{(3)}}Ext^3(\Delta(h),D_{a+k}\otimes\Lambda^{b-k})
	\end{align}
	Using Theorem 3.5 of \cite{MS} and Lemma \ref{basiclemma}, we have the following remark.
	\begin{rem} $\;$\begin{enumerate}
	\item[{\rm (a)}] 	$Ext^1(\Delta(h),\Delta(h(k+1)))=	Ext^1(\Delta(h),D_{a+k}\otimes\Lambda^{b-k})=\mathbb{Z}_2$ 
	and \\
	 $Ext^1(\Delta(h),\Delta(h(k)))=0$, or
	\item[{\rm (b)}] $Ext^1(\Delta(h),D_{a+k}\otimes\Lambda^{b-k})=Ext^1(\Delta(h),\Delta(h(k)))=\mathbb{Z}_2$ 
	and \\
	$Ext^1(\Delta(h),\Delta(h(k+1)))=0$.
\end{enumerate}
 \end{rem}
    Thus (\ref{long1}) takes the following form 
	\begin{align}\label{long2}
	&0\rightarrow Ext^2(\Delta(h),\Delta(h(k+1)))\xrightarrow{i_k^{(2)}}Ext^2(\Delta(h),D_{a+k}\otimes\Lambda^{b-k})\nonumber \\
	&\xrightarrow{\pi_k^{(2)}}Ext^2(\Delta(h),\Delta(h(k)))\longrightarrow Ext^3(\Delta(h),\Delta(h(k+1)))\nonumber \\
	&\xrightarrow{i_k^{(3)}} Ext^3(\Delta(h),D_{a+k}\otimes\Lambda^{b-k})
	\end{align}
	
	Then, Theorem 4.1 by \cite{MS} and (3) yield the case $k=3$ of the theorem.
	\begin{prop}
	We have	$Ext^2(\Delta(h),\Delta(h(3)))=\mathbb{Z}_{3/gcd(3,a+b)}.$
	\end{prop}
	\begin{proof}We observe that $Ext^3(\Delta(h),D_{a+2}\otimes\Lambda^{b-2})=0$, as $Ext^3(\Delta(h),D_{a+2}\otimes\Lambda^{b-2})$ $=Ext^3(\Lambda^3, D_3)$, by Lemma 2.4, and $Ext^3(\Lambda^3, D_3)=0$, as $\Lambda^3$ has a projective resolution of length 2 by \cite{A}. So, for $k=2$, (\ref{long2}) has the following form
	\begin{align}\label{glyko}
	&0\rightarrow Ext^2(\Delta(h),\Delta(h(3)))\xrightarrow{i_2^{(2)}}Ext^2(\Delta(h),D_{a+2}\otimes\Lambda^{b-2})\nonumber \\
	&\xrightarrow{\pi_2^{(2)}}Ext^2(\Delta(h),\Delta(h(2)))\longrightarrow Ext^3(\Delta(h),\Delta(h(3)))\xrightarrow{i_2^{(3)}} 0.
	\end{align}
	Using again Lemma 2.4 we have $Ext^2(\Delta(h),D_{a+2}\otimes\Lambda^{b-2})=\mathbb{Z}_3$, which implies that $Ext^2(\Delta(h),\Delta(h(3)))=\mathbb{Z}_3$ or $0$, as $i_2^{(2)}$ is a monomorphism.
	We also know that $Ext^2(\Delta(h),\Delta(h(2)))=\mathbb{Z}_{d_2}$ and $Ext^3(\Delta(h),\Delta(h(3)))=\mathbb{Z}_{d_3}$, by Theorem 4.1 of \cite{MS}, where $d_2=gcd(a+b,\tbinom{a+b}{2})$ and $d_3=gcd(a+b,\tbinom{a+b}{2},\tbinom{a+b}{3})$.	Simple calculations yield that $d_2=d_3$ if and only if $3\notdivides a+b$, so the exact sequence (\ref{glyko}) implies that $i_2^{(2)}$ is an isomorphism if and only if $3\notdivides a+b$. It follows that $Ext^2(\Delta(h),\Delta(h(3)))=\mathbb{Z}_3$ if and only if $3\notdivides a+b$ and $Ext^2(\Delta(h),\Delta(h(3)))=0$ otherwise.
	\end{proof}
	
	Lemma \ref{basiclemma} (b) also yields that $Ext^2(\Delta(h),D_{a+k-1}\otimes \Lambda^{b-k+1})=0$ for every $k\geq5$ and using again the exact sequence (\ref{long2}) we obtain the following result.
	\begin{prop}
		We have	$Ext^2(\Delta(h),\Delta(h(k)))=0$ for every $5\leq k\leq b$.
	\end{prop} 
	
	\section{The case $k=4$ of the Theorem}
	In this Section we show Theorem 1.1 for $k=4$, which constitutes the main part of the paper. This is done using a refinement of the main strategy of \cite{MS}, namely by first determining a cyclic generator of $Ext^2(\Delta(h),D_{a+3}\otimes \Lambda^{b-3})$, Subsection 4.1, and then computing its image in $Ext^2(\Delta(h),D_{a+2}\otimes \Lambda^{b-2})$, Subsection 4.3, with the aid of the relations that we establish in Subsection 4.2.
	
	In order to compute $Ext^2(\Delta(h),\Delta(h(4)))$, we use the exact sequence (\ref{long2}) for $k=3$, which takes the following form
	\begin{align*}
	&0\longrightarrow
	Ext^2(\Delta(h),\Delta(h(4)))
	\xrightarrow{i_3^{(2)}}Ext^2(\Delta(h),D_{a+3}\otimes \Lambda^{b-3})\\
	&\xrightarrow{\pi_3^{(2)}}
	Ext^2(\Delta(h),\Delta(h(3))).
	\end{align*}
	By Lemma \ref{basiclemma} (b) we have $Ext^2(\Delta(h),D_{a+3}\otimes \Lambda^{b-3})=\mathbb{Z}_3$. 
	In order to determine whether $Ext^2(\Delta(h),\Delta(h(4)))=0$ or $\mathbb{Z}_3$, we consider the following composition  
	\begin{equation*}
	Ext^2(\Delta(h),D_{a+3}\otimes \Lambda^{b-3})\xrightarrow{\pi_3^{(2)}} Ext^2(\Delta(h),\Delta(h(3)))\xrightarrow{i_2^{(2)}} Ext^2(\Delta(h),D_{a+2}\otimes \Lambda^{b-2}).
	\end{equation*}
	Let $\phi=i_2^{(2)}\circ\pi_3^{(2)}$. Then, $\phi $ is a map from $\mathbb{Z}_3$ to $\mathbb{Z}_3$ by Lemma \ref{basiclemma}.
	Our aim in this Section is to find a cyclic generator of $Ext^2(\Delta(h),D_{a+3}\otimes \Lambda^{b-3})$ and determine its image under $\phi$ as a multiple of the cyclic generator of $Ext^2(\Delta(h),D_{a+2}\otimes \Lambda^{b-2})$ described by \cite{MS} in Section 4.1.
	\subsection{Cyclic generator of $Ext^2(\Delta (h),D_{a+3}\otimes \Lambda^{b-3})$.}
	In this subsection we will determine a cyclic generator of the group $Ext^2(\Delta (h),D_{a+3}\otimes \Lambda^{b-3})$. This is accomplished using the recursion of Lemma 2.3 for $i=2$, according to which the matrix $e^{(2)}(\Delta(h),D_{a+3}\otimes \Lambda^{b-3}V)$ of the differential $Hom(\Theta_2(a,b),D_{a+3}\otimes\Lambda^{b-3})$ has the following form
		\[
	\left(
	\begin{array}{c|c|c}
	A &*&*\\
	\hline \;
	
	&B& \\
	\hline
	&&C
	\end{array}
	\right) ,
	\]
	where $A=e^{(1)}(\Delta(a+1,1^{b-1})
	,D_{a+3}\otimes\Lambda^{b-3})$, $B=-e^{(2)}(\Delta(1,1^{b-1}),D_{4}\otimes\Lambda^{b-4})$ and $C=-e^{(2)}(\Delta(1,1^{b-1}),D_{3}\otimes\Lambda^{b-3})$.
	
	Consider the following standard basis elements:
	\begin{itemize}
	\item $T_{i,j}^1=1^{(a)}i^{(2)}j\otimes 2\cdots \hat{i} \cdots \hat{j} \cdots b$ \\
	for $i\in\{2,\dots,b-1\}$, $j\in\{i+1,\dots,b\}$ \\
	in $Hom(D_a\otimes {D_1}^{\otimes (i-2)}\otimes D_2 \otimes {D_1}^{\otimes (b-i)}, D_{a+3}\otimes \Lambda^{b-3})$,
	\item $T_{i,j}^k=1^{(a-1)}ki^{(2)}j\otimes 1\cdots\hat{k}\cdots \hat{i} \cdots \hat{j} \cdots b$ \\
	for $k\in\{2,\dots,b-2\}$,
	$i\in\{k+1,\dots,b-1\}$, $j\in\{i+1,\dots,b\}$ \\
	in $Hom(D_a\otimes {D_1}^{\otimes (i-2)}\otimes D_2 \otimes {D_1}^{\otimes (b-i)},D_{a+3}\otimes \Lambda^{b-3})$,
	\end{itemize}
    and the standard basis elements:
    \begin{itemize}
    \item $B_{1,j}^{a,b}=1^{(a+2)}j\otimes 2\cdots \hat{j} \cdots (b-1)$ for $j\in\{2,\dots ,b-1\}$, \\
    in $Hom(D_{a+2}\otimes {D_1}^{\otimes (b-2)},D_{a+3}\otimes \Lambda^{b-3})$, ,
    \item $B_{i,1}^{a,b}=1^{(a)}i^{(3)}\otimes 2\cdots\hat{i}\cdots(b-1)$ for $i\in\{2,\dots,b-1\}$, \\
    in $Hom(D_a\otimes {D_1}^{\otimes (i-2)}\otimes D_3\otimes {D_1}^{\otimes (b-i-1)}, D_{a+3} \otimes \Lambda^{b-3})$,
    \item $B_{i,j}^{a,b}=1^{(a-1)}i^{(3)}j\otimes 1\cdots \widehat{\{i,j\}} \cdots (b-1)$  for $i\in\{2,\dots,b-1\}$ \\ and $j\in\{2,\dots,i-1\}\cup\{i+1,\dots,b-1\}$, \\
    in  $Hom(D_a\otimes {D_1}^{\otimes (i-2)}\otimes D_3\otimes {D_1}^{\otimes (b-i-1)}, D_{a+3} \otimes \Lambda^{b-3})$.
    \end{itemize}

Recall that the cokernel of the differential $Hom(\Theta_i (a,b),M)$ of the complex $Hom(P_{*} (a,b),M)$ is denoted by $E^i (\Delta(h),M)$ (cf. 2.4).

Let $\pi: Hom(P_i (a,b),M)\rightarrow E^{i}(\Delta(h),M)$ be the natural projection.

\smallskip
\noindent \textit{Remark.} In the following proposition and its proof, if the lower bound of the summation index is less than the upper bound, then we regard the sum as zero (empty sum).

\begin{prop} Let 
\begin{align*}\Gamma_{a,b}=&\tbinom{a+2} {3}\sum_{j=2}^{b-1}{(-1)}^{j}B_{1,j}^{a,b} -\Big(aB_{2,1}^{a,b}+\sum_{j=3}^{b-1}{(-1)}^{j}B_{2,j}^{a,b}\Big)+ \\
&\sum_{i=3}^{b-2}{(-1)}^{i-1}\Big(aB_{i,1}^{a,b} -\sum_{j=2}^{i-1} {(-1)}^j B_{i,j}^{a,b}+\sum_{j=i+1}^{b-1} {(-1)}^{j} B_{i,j}^{a,b}  \Big)+\\
&{(-1)}^{b}\Big( aB_{b-1,1}^{a,b}-\sum_{j=2}^{b-2}{(-1)}^jB_{b-1,j}^{a,b}\Big).
\end{align*} 
Then, for $b\geq 3$, $3\pi(\Gamma_{a,b})=0$ in $E^{2}(\Delta(h),D_{a+3}\otimes \Lambda^{b-3})$ and thus $\pi(\Gamma_{a,b})$ is an element of $Ext^2(\Delta(h),D_{a+3}\otimes \Lambda^{b-3})$.          
\end{prop}

\begin{proof} Let \[A_{a,b}=a\sum_{i=2}^{b-1}\sum_{j=i+1}^{b} {(-1)}^{j-i+1}\sum_{s=1}^{b-1}{(-1)}^{s-1}\theta_s(T_{i,j}^1),\] \[C_{a,b}=\sum_{k=2}^{b-2} \sum_{i=k+1}^{b-1} \sum_{j=i+1}^{b} {(-1)}^{j-i-k} \sum_{s=1}^{b-1}{(-1)}^{s-1}\theta_s (T_{i,j}^k).\] We will use induction on $b\geq3$ to show that 
	\begin{align}\label{eqnionia}
     A_{a,b}+C_{a,b}=3\Gamma_{a,b}
	\end{align}
	For $b=3$ we observe that $C_{a,3}=0$ as it is an empty sum, and for $T_{2,3}^1=1^{(a)}2^{(2)}3$ in $Hom(D_a\otimes D_2\otimes D_1,D_{a+3})$ we have that \begin{align*}&A_{a,3}+C_{a,3}=A_{a,b}=a(\theta_1(T_{2,3})-\theta_1(T_{2,3}))=a\big({\tbinom{a+2}{2}}\cdot1^{(a+2)}2-3\cdot 1^{(a)}2^{(3)}\big)\\
	&=3{\tbinom{a+2}{3}}\cdot B_{1,2}^{a,3}-3a\cdot B_{2,1}^{a,3}=3\Gamma_{a,3}.
	\end{align*}
    Let $b>3$. For $i=2$ and $j\in\{3,\dots,b\}$ we observe that  
	$$\theta_1 (T_{2,j}^1)=\tbinom{a+2}{2} 1^{(a+2)}(j-1)\otimes 2\cdots \widehat{(j-1)} \cdots (b-1)=\tbinom{a+2}{2}B_{1,j-1}^{a,b},$$
	and so we have 
	\begin{align} \label{a1}
	A_{a,b}&=a\sum_{i=2}^{b-1}\sum_{j=i+1}^{b} {(-1)}^{j-i+1}\theta_1(T_{i,j}^1) \nonumber \\
	&\,\,\,\,\,\,\,+a\sum_{i=2}^{b-1}\sum_{j=i+1}^{b} {(-1)}^{j-i+1}\sum_{s=2}^{b-1}{(-1)}^{s-1}\theta_s(T_{i,j}^1)\nonumber\\
	&=a\tbinom{a+2}{2}\sum_{j=3}^{b}{(-1)}^{j-1}B_{1,j-1}^{a,b}+
	a\sum_{i=3}^{b-1}\sum_{j=i+1}^{b} {(-1)}^{j-i+1}\theta_1(T_{i,j}^1)\nonumber \\
	&\,\,\,\,\,\,\,+a\sum_{i=2}^{b-1}\sum_{j=i+1}^{b} {(-1)}^{j-i+1}\sum_{s=2}^{b-1}{(-1)}^{s-1}\theta_s(T_{i,j}^1).
	\end{align}
	Consider now the standard basis elements: 
	\begin{itemize}
	\item $t_{1,j-1}^0 = 1^{(2)}(j-1)\otimes 2\cdots \widehat{(j-1)} \cdots (b-1)$ in \\
	$Hom(D_2\otimes {D_1}^{\otimes (b-1)},D_3 \otimes \Lambda^{b-1})$, for $i=2$ and $j\in\{3,\dots,b\}$,
	\item $t_{i-1,j-1}^0={(i-1)}^{(2)}(j-1)\otimes 1\cdots \widehat{(i-1)} \cdots \widehat{(j-1)} \cdots (b-1)$ in \\
	$Hom({D_1}^{\otimes (i-2)}$ $\otimes D_2\otimes {D_1}^{\otimes (b+2-i)},D_3 \otimes \Lambda^{b-1})$, for $i\in\{3,\dots,b-1\}$ \\
	and $j\in\{i+1,\dots,b\}$.
	\end{itemize}
	We observe that $\theta_{s-1}(t^0_{i-1,j-1})=f_2(\theta_s(T^1_{i,j}))$ for every $i\in\{2,\dots,b-1\}$, $j\in\{i+1,\dots,b\}$ and $s\in\{2,\dots,b-1\}$, where $f_2$ is the isomorphism described in Lemma \ref{skewmatrix} for $i=2$. By \cite{MS}, Lemma 4.4 for $k=2$, we have the following relation in $Ext^2(\Delta(1,1^{(b-1)}),D_3 \otimes \Lambda^{b-3})$
	\begin{equation}\label{eqms}
	Hom(\Theta_2(a,b),D_{a+2}\otimes \Lambda^{b-2})(\sum_{i=2,j>i}^{b-1,b}{(-1)}^{j-i+1}t_{i-1,j-1}^0)=3\gamma_{1,b-1},
	\end{equation}
	where $\gamma_{1,b-1}=\sum_{i=2}^{b-1}{(-1)}^i\delta_{i-1}^{1,b-1}$,
	\begin{itemize}
		\item $\delta_1^{1,b-1}=1^{(3)}\otimes 2\cdots (b-2)$ in $Hom(D_3\otimes {D_1}^{\otimes (b-3)},D_3\otimes\Lambda^{b-3})$, for $i=2$, 
		\item $\delta_{i-1}^{1,b-1}={(i-1)}^3\otimes 1\cdots \widehat{(i-1)} \cdots (b-2)$ in $Hom({D_1}^{\otimes (i-2)}\otimes D_3\otimes {D_1}^ {\otimes (b-i-1)},$ $D_3\otimes\Lambda^{b-3})$, for $i\in\{3,\dots,b-1\}$,
	\end{itemize}
and $\pi(\gamma_{1,b-1})$ is the generator of $Ext^2(\Delta(1,1^{b-1}),D_3\otimes \Lambda^{b-3})$ described by \cite{MS} in Subsection 4.1.
Using Remark 2.5, (\ref{eqms}) is equivalent to
	\begin{equation}\label{rel}
\sum_{i=2}^{b-1}\sum_{j=i+1}^{b} {(-1)}^{j-i+1}\sum_{s=2}^{b-1}{(-1)}^{s}\theta_{s-1}(t_{i-1,j-1}^0)              =3\sum_{i=2}^{b-1}{(-1)}^i\delta_{i-1}^{1,b-1}.
\end{equation} 
By the proof of Lemma \ref{skewmatrix}, we have that $\delta_{i-1}^{1,b-1}=f_2(B_{i,1}^{a,b})$ where
\begin{itemize}
	\item $B_{2,1}^{a,b}=1^{(a)}2^{(3)}\otimes 3\cdots (b-1)$ for $i=2$,
	\item $B_{i,1}^{a,b}=1^{(a)}i^{(3)}\otimes 2\cdots \hat{i} \cdots (b-1)$ for $i\in\{3,\dots,b-1\}$.
\end{itemize}
and $f_2$ is the isomorphism described in Lemma \ref{skewmatrix} for $i=2$. 

Applying the isomorphism $f_2^{-1}$ to Equation (\ref{rel}) we have, 
\begin{equation*}
 \sum_{i=2}^{b-1}\sum_{j=i+1}^{b} {(-1)}^{j-i+1}\sum_{s=2}^{b-1}{(-1)}^{s}f_2^{-1}(\theta_{s-1}(t_{i-1,j-1}^0))              =3\sum_{i=2}^{b-1}{(-1)}^i f_2^{-1}(\delta_{i-1}^{1,b-1})
\end{equation*}
or, equivalently,
\begin{equation}\label{e14e}
 \sum_{i=2}^{b-1}\sum_{j=i+1}^{b} {(-1)}^{j-i+1}\sum_{s=2}^{b-1}{(-1)}^{s}\theta_s(T_{i,j}^1)             =3\sum_{i=2}^{b-1}{(-1)}^iB_{i,1}^{a,b}.
\end{equation}
Using (\ref{e14e}), (\ref{a1}) is equivalent to
\begin{align}\label{a4}
A_{a,b}&=a\tbinom{a+2}{2}\sum_{i=2}^{b-1}{(-1)}^{i}B_{1,i}^{a,b}+a\sum_{i=3}^{b-1}\sum_{j=i+1}^{b} {(-1)}^{j-i+1}\theta_1(T_{i,j}^1)\nonumber\\
&\,\,\,\,\,\,-3a\sum_{i=2}^{b-1}{(-1)}^i B_{i,1}^{a,b}.
\end{align}
We also have
\begin{align}\label{a8}
C_{a,b}&=\sum_{i=3}^{b-1} \sum_{j=i+1}^{b} {(-1)}^{j-i}\theta_1 (T_{i,j}^2) + \sum_{k=3}^{b-2} \sum_{i=k+1}^{b-1} \sum_{j=i+1}^{b} {(-1)}^{j-i-k}\theta_1 (T_{i,j}^k)\nonumber \\
&\,\,\,\,\,\,+\sum_{k=2}^{b-2} \sum_{i=k+1}^{b-1} \sum_{j=i+1}^{b} {(-1)}^{j-i-k} \sum_{s=2}^{b-1}{(-1)}^{s-1}\theta_s (T_{i,j}^k)\nonumber \\
&=\sum_{i=3}^{b-1} \sum_{j=i+1}^{b} {(-1)}^{j-i}\theta_1 (T_{i,j}^2) +\sum_{k=2}^{b-2} \sum_{i=k+1}^{b-1} \sum_{j=i+1}^{b} {(-1)}^{j-i-k} \sum_{s=2}^{b-1}{(-1)}^{s-1}\theta_s (T_{i,j}^k),
\end{align}	
as $\theta_1 (T_{i,j}^k)=0$ for every $k\in\{3\dots,b-2\}$, because of the $12$ part of $T_{i,j}^k$ in $\Lambda^{b-3}$. 
Consider now the basis elements: 
\begin{itemize}
	\item $t_{i-1,j-1}^1=1{(i-1)}^{(2)}(j-1)\otimes 2\cdots\widehat{(i-1)}\cdots\widehat{(j-1)}\cdots (b-1)$ in \\
	$Hom({D_1}^{\otimes (i-2)}\otimes D_2 \otimes {D_1}^{\otimes (b-i)},D_4\otimes\Lambda^{b-4})$, for $k=2$ and $i\in\{3,\dots,b-1\}$, $j\in\{i+1,\dots,b\}$,
	\item $t_{i-1,j-1}^{k-1}=(k-1){(i-1)}^{(2)}(j-1)\otimes 1\cdots\widehat{(k-1)}\cdots\widehat{(i-1)}\cdots \widehat{(j-1)} \cdots (b-1)$ in $Hom({D_1}^{\otimes (i-2)}\otimes D_2 \otimes {D_1}^{\otimes (b-i)},D_4\otimes\Lambda^{b-4})$ , for $k\in\{3,\dots,b-2\}$, $i\in\{k+1,\dots,b-1\}$, $j\in\{i+1,\dots,b\}$.
\end{itemize}
 By induction on $b$, for $a=1$, we have  $A_{1,b-1}+C_{1,b-1}=3\Gamma_{1,b-1},$ or, equivalently,
\begin{align}\label{eq10}
&\sum_{k=2}^{b-2} \sum_{i=k+1}^{b-1} \sum_{j=i+1}^{b} {(-1)}^{j-i-k+1} \sum_{s=2}^{b-1}{(-1)}^{s}\theta_{s-1} (t_{i-1,j-1}^{k-1})\nonumber \\
&=3\sum_{j=3}^{b-1}{(-1)}^{j-1}B_{1,j-1}^{1,b-1} -3\Big(B_{2,1}^{1,b-1}+\sum_{j=4}^{b-1}{(-1)}^{j-1}B_{2,j-1}^{1,b-1}\Big)\nonumber \\
&\,\,\,\,\,\,+3\sum_{i=4}^{b-1}{(-1)}^{i}\Big(B_{i-1,1}^{1,b-1} -\sum_{j=3}^{i-1} {(-1)}^{j-1} B_{i-1,j-1}^{1,b-1}+\sum_{j=i+1}^{b-1} {(-1)}^{j-1} B_{i-1,j-1}^{1,b-1}  \Big)\nonumber \\
&\,\,\,\,\,\,+3{(-1)}^{b-1}\Big(B_{b-2,1}^{1,b-1}-\sum_{j=3}^{b-2}{(-1)}^{j-1}B_{b-2,j-1}^{1,b-1}\Big).
\end{align}
Applying $f_2^{-1}$ to (\ref{eq10}), using that $f_2(B_{i,j}^{a,b})=B_{i-1,j-1}^{1,b-1}$ and $f_2(\theta_s(T_{i,j})^k)=\theta_{s-1}(t_{i-1,j-1}^{k-1})$ for $i\geq 2$, $j\in\{3,\dots,i-1\}\cup\{i+1,\dots,b-2\}$, we obtain 
\begin{align}\label{eq11}
&\sum_{k=2}^{b-2} \sum_{i=k+1}^{b-1} \sum_{j=i+1}^{b} {(-1)}^{j-i-k+1} \sum_{s=2}^{b-1}{(-1)}^{s}\theta_{s} (T_{i,j}^{k})\nonumber \\
&=3\sum_{j=3}^{b-1}{(-1)}^{j-1}B_{2,j}^{a,b} -3\Big(B_{3,2}^{a,b}+\sum_{j=4}^{b-1}{(-1)}^{j-1}B_{3,j}^{a,b}\Big)\nonumber \\
&\,\,\,\,\,\,+3\sum_{i=4}^{b-2}{(-1)}^{i}\Big(B_{i,2}^{a,b} -\sum_{j=3}^{i-1} {(-1)}^{j-1} B_{i,j}^{a,b}+\sum_{j=i+1}^{b-1} {(-1)}^{j-1} B_{i,j}^{a,b}  \Big)\nonumber \\
&\,\,\,\,\,\,+3{(-1)}^{b-1}\Big(B_{b-1,2}^{a,b}-\sum_{j=3}^{b-2}{(-1)}^{j-1}B_{b-1,j}^{a,b}\Big).
\end{align}
Equation (\ref{eq11}) implies that (\ref{a8}) is equivalent to
\begin{align}\label{eq12}
C_{a,b}&=\sum_{i=3}^{b-1} \sum_{j=i+1}^{b} {(-1)}^{j-i}\theta_1 (T_{i,j}^2) +3\sum_{j=3}^{b-1}{(-1)}^{j-1}B_{2,j}^{a,b} \nonumber \\
&\,\,\,\,\,\,-3\Big(B_{3,2}^{a,b}+\sum_{j=4}^{b-1}{(-1)}^{j-1}B_{3,j}^{a,b}\Big)\nonumber \\
&\,\,\,\,\,\,+3\sum_{i=3}^{b-2}{(-1)}^{i}\Big(B_{i,2}^{a,b} -\sum_{j=3}^{i-1} {(-1)}^{j-1} B_{i,j}^{a,b}+\sum_{j=i+1}^{b-1} {(-1)}^{j-1} B_{i,j}^{a,b}  \Big)\nonumber \\
&\,\,\,\,\,\,+3{(-1)}^{b-1}\Big(B_{b-1,2}^{a,b}-\sum_{j=3}^{b-2}{(-1)}^{j-1}B_{b-1,j}^{a,b}\Big).
\end{align}
We observe that $\theta_1(T_{i,j}^2)=a\theta_1(T_{i,j}^1)$ for $i\in\{3,\dots,b-1\}$, $j\in\{i+1,\dots,b\}$, as $\theta_1(T_{i,j}^1)=1^{(a)}{(i-1)}^2(j-1)\otimes 1\cdots\widehat{(i-1)}\cdots \widehat{(j-1)}\cdots (b-1)$, and
$\theta_1(T_{i,j}^2)=a1^{(a)}{(i-1)}^2(j-1)\otimes 1\cdots\widehat{(i-1)}\cdots \widehat{(j-1)}\cdots (b-1)$. Using this observation and adding equations  (\ref{a4}) and (\ref{eq12}), we obtain $A_{a,b}+C_{a,b}=3\Gamma_{a,b}$
and we have proved inductively equation (\ref{eqnionia}).

As we have the relations $\sum_{s=1}^{b-1}{(-1)}^{s-1}\pi(\theta_s(T_{i,j}^k))=0$ for every $k\in\{1,\dots,b-2\}$, $i\in\{k+1,\dots,b-1\}$, $j\in\{i+1,\dots,b\}$, it follows that $\pi(A_{a,b})=\pi(C_{a,b})=0$ and so $3\pi(\Gamma_{a,b})=0$.\end{proof}
\begin{prop}
The element $\pi(\Gamma_{a,b})$ of Proposition 3.1 is a cyclic generator of  $Ext^2(\Delta(h),D_{a+3}\otimes\Lambda^{b-3})$.
\end{prop}	
\begin{proof}We will first show that $\pi(\Gamma_{a,b})$ is nonzero. Indeed, if $\pi(\Gamma_{a,b})=0$, then $\pi(\Gamma_{a,b})$ would be equal to a linear combination of the relations in the extension group $Ext^2(\Delta(h),D_{a+3}\otimes\Lambda^{b-3})$ corresponding to the columns of the matrix $e^{(2)}(\Delta(h),D_{a+3}\otimes\Lambda^{b-3})$. This implies that the coefficient $-1$ of $\pi(B_{b-1,b-2})$ in $\pi(\Gamma_{a,b})$ should be a linear combination of the nonzero entries of the last row of $e^{(2)}(\Delta(h),D_{a+3}\otimes\Lambda^{b-3})$. These are the coefficients of $\pi(B_{b-1,b-2})$ resulting from the application of $\pi\circ\theta_{b-1}$ on the standard basis elements: $$1^{(a-1)} (b-2) {(b-1)}^{(2)} b\otimes 1\cdots(b-3), \, 1^{(a-1)}(b-2)(b-1)b^{(2)}\otimes 1\cdots(b-3).$$ Indeed, let $\lambda=(\lambda_1,\lambda_2,\dots,\lambda_b)$ be a sequence with $\lambda_1\in\{0,1\},\,\lambda_i\in\{1,2\},\, \sum_{i=1}^{b}\lambda_i$ $=b$, and a standard basis element $T_{\lambda}$ in $Hom(D_{a+\lambda_1}\otimes D_{\lambda_2}\otimes \cdots \otimes  D_{\lambda_b}, D_{a+3}\otimes\Lambda^{b-3})$, such that $\theta_{b-1}(T_{\lambda})=c\cdot B_{b-1,b-2}^{a,b}$ for some $c\in\mathbb{Z}$. Then $\theta_{b-1}(T_{\lambda})$ has weight $(1^{a+\lambda_1},2^{\lambda_2},\dots,{(b-1)}^{\lambda_{b-1}+\lambda_b})$ which must be equal to $(1^{a},2,\dots,(b-2),{(b-1)}^3)$, the weight of $B_{b-1,b-2}^{a,b}$. This implies that $\lambda_1=0$, $(\lambda_{b-1},\lambda_{b})\in\{(1,2),(2,1)\}$ and $$T_{\lambda}\in\{1^{(a-1)} (b-2) {(b-1)}^{(2)} b\otimes 1\cdots(b-3), 1^{(a-1)}(b-2)(b-1)b^{(2)}\otimes 1\cdots(b-3)\}.$$ By straightforward calculations, we have, 
	\begin{align*}
	&\pi(\theta_{b-1}(1^{(a-1)} (b-2) {(b-1)}^{(2)} b\otimes 1\cdots(b-3)))\\
	&=\pi(\theta_{b-1}(1^{(a-1)}(b-2)(b-1)b^{(2)}\otimes 1\cdots(b-3)))\\
	&=3 \pi(B_{b-1,b-2}^{a,b}),
	\end{align*} 
	which means that $-1$ should be a linear combination of $3$ and $3$, i.e. multiple of $3$. This implies that $\pi(\Gamma_{a,b})$ is nonzero. As $Ext^2(\Delta(h),D_{a+3}\otimes\Lambda^{b-3})= \mathbb{Z}_3$, $3\pi(\Gamma_{a,b})=0$ (by Proposition 4.1) and $\pi(\Gamma_{a,b})\neq 0$, we conclude that $\pi(\Gamma_{a,b})$ is a generator of $Ext^2(\Delta(h),D_{a+3}\otimes\Lambda^{b-3})$. \end{proof}
\subsection{Relations in $Ext^2(\Delta(h),D_{a+2}\otimes \Lambda^{b-2})$.}
In this Subsection we will prove certain relations in $E^2(\Delta(h),D_{a+2}\otimes \Lambda^{b-2})$ using the columns of $e^{(2)}(\Delta(h), D_{a+2}\otimes \Lambda^{b-2})$, i.e. using the equalities $\sum_{s=1}^{b-1}{(-1)}^{s-1}\pi(\theta_s(T))=0$, where $T$ is a standard basis element of the domain of the differential $\Theta_2 (a,b)$. These relations will be described in Lemmas 4.3 and 4.4. They will yield Proposition 4.5 which will be used in the proof of Proposition 4.6, in order to facilitate the computation of the image of the cyclic generator of Proposition 4.2. Proposition 4.6 will then yield the case $k=4$ of Theorem 1.1.

Let  \begin{align*}{b_1}^{(1)}&=1^{(a+2)} \otimes 2\cdots (b-1),\\{b_j}^{(1)}&=1^{(a+1)}j \otimes 1\cdots \hat{j} \cdots (b-1)\end{align*} 
be the indicated standard basis elements in $Hom(D_{a+2}\otimes {D_1}^{\otimes (b-2)}, D_{a+2} \otimes \Lambda^{b-2})$ and 
\begin{align*}{b_1}^{(i)}=&1^{(a)}i^{(2)} \otimes 2 \cdots i \cdots (b-1),\\{b_j}^{(i)}=&1^{(a-1)} i^{(2)}j \otimes 12\cdots \hat{j} \cdots (b-1)\end{align*} 
be the indicated standard basis elements in $Hom(D_a \otimes {D_1}^{\otimes (i-2)} \otimes D_3 \otimes {D_1}^{\otimes (b-i-1)},$ $ D_{a+2}\otimes \Lambda^{b-2})$ for $i\in\{2,\dots ,b-1\}$ and $j\in\{2,\dots,b-1\}$.
\begin{lm} \label{bl1} For every $i\in\{2,\dots,b-1\}$,
		\begin{equation*}
		a\pi(b_1^{(i+1)}) - \sum_{j=2}^{i}{(-1)}^j \pi(b_j^{(i+1)})= a\pi(b_1^{(i)})-\sum_{j=2}^{i-1}{(-1)}^{(j)}\pi(b_j^{(i)}) +3{(-1)}^{i+1}\pi(b_i^{(i)}).
		\end{equation*}
\end{lm}
\begin{proof} We consider the basis elements $T_1^{(i+1)}=1^{(a)}{(i+1)}^{(2)}\otimes 2\cdots \widehat{(i+1)} \cdots b$ and 
	$T_j^{(i+1)}=1^{(a-1)}j{(i+1)}^{(2)}\otimes 1\cdots \hat{j}\cdots \widehat{(i+1)}\cdots b$, for $j\in\{2,\dots,i\}$. Let \[D_{a,b}=a\sum_{s=1}^{b-1} {(-1)}^{s-1} \theta_s(T_1^{(i+1)})+\sum_{j=2}^i {(-1)}^j\sum_{s=1}^{b-1} {(-1)}^{s-1} \theta_s(T_j^{(i+1)}).\]
	We know that $\sum_{s=1}^{b-1} {(-1)}^{s-1}\pi( \theta_s(T_j^{(i+1)}))=0$, for $j\in\{1,\dots,i\}$, which implies that $\pi(D_{a,b})=0$.
	We observe that $\theta_s (T_j^{(i+1)})=0$ for every $j\in\{1,\dots,i\}$ and $s\in\{i+2,\dots,b-1\}$ because of the $(i+2)(i+3)\cdots b$ part $T_j^{(i+1)}$ in $\Lambda^{b-2}$. So we obtain 
	\begin{equation*} 
		a\sum_{s=i+2}^{b-1} {(-1)}^{s-1} \theta_s(T_1^{(i+1)})+\sum_{j=2}^i {(-1)}^j\sum_{s=i+2}^{b-1} {(-1)}^{s-1} \theta_s(T_j^{(i+1)})=0
	\end{equation*}
	and
	\begin{equation}\label{2e}
	D_{a,b}=a\sum_{s=1}^{i+1} {(-1)}^{s-1} \theta_s(T_1^{(i+1)})+\sum_{j=2}^i {(-1)}^j\sum_{s=1}^{i+1} {(-1)}^{s-1} \theta_s(T_j^{(i+1)}).
	\end{equation}
	By immediate calculations we have $\theta_s(T_1^{(i+1)})=0$ for every $s\in\{2,\dots,i-1\}$, $\theta_s(T_j^{(i+1)})=0$ for every $s\in\{1,\dots,j-2\}\cup\{j+1,\dots,i-1\}$, $j\in\{2,\dots,i-1\}$ and
	$\theta_s(T_i^{(i+1)})=0$ for every $s\in\{1,\dots,i-2\}$.
    Hence
    \begin{align} D_{a,b}=&a\sum_{s=1}^{i-1} {(-1)}^{s-1} \theta_s(T_1^{(i+1)})+\sum_{j=2}^i {(-1)}^j\sum_{s=1}^{i-1} {(-1)}^{s-1} \theta_s(T_j^{(i+1)})\nonumber \\
    &+a\sum_{s=i}^{i+1} {(-1)}^{s-1} \theta_s(T_1^{(i+1)})+\sum_{j=2}^i {(-1)}^j\sum_{s=i}^{i+1} {(-1)}^{s-1} \theta_s(T_j^{(i+1)})\nonumber \\
    =& (a\theta_1(T_1^{(i+1)})-\theta_1(T_2^{(i+1)})) + \sum_{s=2}^{i-1}(\theta_s(T_s^{(i+1)})-\theta_s(T_{s+1}^{(i+1)}))\nonumber \\
    &+a\sum_{s=i}^{i+1} {(-1)}^{s-1} \theta_s(T_1^{(i+1)})+\sum_{j=2}^i {(-1)}^j\sum_{s=i}^{i+1} {(-1)}^{s-1} \theta_s(T_j^{(i+1)})\nonumber\\
    =&a\sum_{s=i}^{i+1} {(-1)}^{s-1} \theta_s(T_1^{(i+1)})+\sum_{j=2}^i {(-1)}^j\sum_{s=i}^{i+1} {(-1)}^{s-1} \theta_s(T_j^{(i+1)}),
    \end{align}
 as $a\theta_1(T_1^{(i+1)})=\theta_1(T_2^{(i+1)})=a1^{(a)}i^{(2)}\otimes 1\cdots \hat{i} \cdots (b-1)$ and $\theta_s(T_s^{(i+1)})=\theta_s(T_{s+1}^{(i+1)})=1^{(a-1)}si^{(2)}\otimes 1\cdots\hat{i} \cdots (b-1)$. 
 
 It follows that
\begin{equation*}
\pi(D_{a,b})=a\sum_{s=i}^{i+1} {(-1)}^{s-1} \pi(\theta_s(T_1^{(i+1)}))+\sum_{j=2}^i {(-1)}^j\sum_{s=i}^{i+1} {(-1)}^{s-1}\pi(\theta_s(T_j^{(i+1)}))=0,
\end{equation*}
	which is equivalent to
	\begin{equation} \label{6e}
	a\pi(\theta_i(T_1^{(i+1)}))-\sum_{j=2}^{i}{(-1)}^j\pi(\theta_i(T_j^{(i+1)})) =
	a\pi(\theta_{i+1}(T_1^{(i+1)}))-\sum_{j=2}^{i}{(-1)}^j\pi(\theta_{i+1}(T_j^{(i+1)})).
	\end{equation}
	More straightforward calculations yield:
	\begin{itemize}
		\item $\theta_i(T_1^{(i+1)})=1^{(a)}i^{(2)}\otimes 2\cdots (b-1)=b_1^{(i)}$,
		\item $\theta_i(T_j^{(i+1)})=1^{(a-1)}ji^{(2)}\otimes 1\cdots \hat{j} \cdots (b-1)=b_j^{(i)}$ for every $j\in\{2,\dots ,i-1\}$,
		\item $\theta_i(T_i^{(i+1)})=3\cdot 1^{(a-1)}i^{(3)}\otimes 1\cdots \hat{i} \cdots (b-1)=3 \cdot b_i^{(i)}$
	\item $\theta_{i+1}(T_1^{(i+1)})=1^{(a)}{(i+1)}^{(2)}\otimes 2\cdots (b-1) = b_1^{(i+1)}$,
	\item 	$\theta_{i+1}(T_j^{(i+1)})=1^{(a-1)}j{(i+1)}^2\otimes 1 \cdots \hat{j} \cdots (b-1)=b_j^{(i+1)}$ for every \\
	$j\in\{2,\dots ,i\}$,
\end{itemize}
so (\ref{6e}) is equivalent to 
\begin{equation*}
a\pi(b_1^{(i+1)}) - \sum_{j=2}^{i}{(-1)}^j \pi(b_j^{(i+1)})= a\pi(b_1^{(i)})-\sum_{j=2}^{i-1}{(-1)}^{j}\pi(b_j^{(i)}) +3{(-1)}^{i+1}\pi(b_i^{(i)}).
\end{equation*}
\end{proof}
With the notation established before Lemma 4.3, we have the following statement.
\begin{lm} \label{1q1q1} For every $i=2,\dots ,b-2$, \[\sum_{j=i+2}^{b-1}{(-1)}^{j-i}\pi(b_j^{(i+1)})-3\pi(b_{i+1}^{(i+1)})=\sum_{j=i+1}^{b-1}{(-1)}^{j-i}\pi(b_j^{(i)}).\]
\end{lm}
\begin{proof} Consider the standard basis elements $$T_j^{(i+1)}=1^{(a-1)}(i+1)^{(2)}j\otimes 1\cdots \widehat{(i+1)} \cdots \hat{j} \cdots b$$ 
in $Hom(D_a \otimes {D_1}^{\otimes (i-1)}\otimes D_2 \otimes {D_1}^{\otimes (b-i-1)}, D_{a+2} \otimes \Lambda^{b-2})$ 
for $j\in\{i+2,\dots,b\}$. We know that $\sum_{s=1}^{b-1}{(-1)}^{s-1}\pi(\theta (T_j^{(i+1)}))=0$ for $j\in\{i+2,\dots,b\}$. This implies that
\begin{equation}\label{t1}
\sum_{j=i+2}^b {(-1)}^{j}\sum_{s=1}^{b-1}{(-1)}^{s-1}\pi(\theta_s(T_j^{(i+1)}))=0.
\end{equation}
But $\theta_s(T_j^{(i+1)})=0$ for every $s\in\{1,\dots , i-1\}$, for every $j\in\{i+2 ,\dots ,b\}$ because of the $1\cdots i$ part of $T_j^{(i+1)}$
in $\Lambda^{b-2}$. So (\ref{t1}) is equivalent to
\begin{equation}\label{t2}
\sum_{j=i+2}^b {(-1)}^{j}\sum_{s=i}^{i+1}{(-1)}^{s-1}\pi(\theta_s(T_j^{(i+1)}))+
\sum_{j=i+2}^b {(-1)}^{j}\sum_{s=i+2}^{b-1}{(-1)}^{s-1}\pi(\theta_s(T_j^{(i+1)}))=0.
\end{equation}
By straightforward calculations we have  
$\theta_s(T_{i+2}^{(i+1)})=0$ for every $s\in\{i+3,\dots,b-1\}$,
$\theta_s(T_j^{(i+1)})=0$ for every $s\in\{i+2,\dots ,j-2\}\cup \{j+1, \dots ,b-1\}$ for every $j\in\{i+3, \dots , b-1\}$ and $\theta_s(T_b^{(i+1)})=0$ for every $s\in\{i+2, \dots ,b-2\}$.
So (\ref{t2}) yields
\begin{align}\label{t3}
\sum_{j=i+2}^b {(-1)}^{j}\sum_{s=i}^{i+1}{(-1)}^{s-1}\pi(\theta_s(T_j^{(i+1)}))-
\sum_{j=i+2}^{b-1}(\pi(\theta_j(T_j^{(i+1)}))-\pi(\theta_j(T_{j+1}^{(i+1)})))=0.
\end{align}
We observe now that $\theta_j(T_j^{(i+1)})=\theta_j(T_{j+1}^{(i+1)})$, so (\ref{t3}) is equivalent to 
$$\sum_{j=i+2}^b {(-1)}^{j-i}(\pi(\theta_{i+1}(T_j^{(i+1)}))-\pi(\theta_i(T_j^{(i+1)})))=0,$$
or
\begin{equation}\label{t6}
\sum_{j=i+2}^{b-1}{(-1)}^{j-i} \pi(\theta_i(T_j^{(i+1)}))=
\sum_{j=i+3}^{b-1}{(-1)}^{j-i} \pi(\theta_{i+1}(T_j^{(i+1)})) + \pi(\theta_{i+1}(T_{i+2})),
\end{equation}
where 
\begin{itemize}
\item $\theta_i(T_j^{(i+1)})=b_{j-1}^{(i)}$ for every $j\in\{i+2,\dots ,b\}$,
\item $\theta_{i+1}(T_{i+2})= 3\cdot b_{i+1}^{(i+1)}$ and $\theta_{i+1}(T_j^{(i+1)})=b_{j-1}^{(i+1)}$ for every $j\in\{i+3, \dots ,b\}$.
\end{itemize}
\noindent Equation (\ref{t6}) is now equivalent to
\begin{equation*}
\sum_{j=i+2}^{b-1}{(-1)}^{j-i}\pi(b_j^{(i+1)})-3\pi(b_{i+1}^{(i+1)})=\sum_{j=i+1}^{b-1}{(-1)}^{j-i}\pi(b_j^{(i)})
\end{equation*}
\noindent for every $i\in\{2,\dots ,b-2\}$. \end{proof}
We are now ready to prove the main result of the present subsection, which provides usefull relations that will be used in Subsection 4.3.
\begin{prop} \label{l1} $\;$ \begin{enumerate}
			\item[{\rm (a)}] For $i=1$, \[\sum_{j=2}^{b-1} (-1)^j \pi(b_j^{(1)})= (a+2)\pi(b_1^{(1)}).\]
	       \item[{\rm (b)}] For every $i\in\{2,\dots, b-1\}$, \[a(-1)^i \pi(b_1^{(i)})+ (-1)^{i}\sum_{j=2}^{i-1} (-1)^{j-1} \pi(b_j^{(i)})+{(-1)}^i \sum_{j=i+1}^{b-1} (-1)^{j-1} \pi(b_j^{(i)}) = 3\pi(b_i^{(i)}).\]
		 \end{enumerate} \end{prop}
\begin{proof} For $i=1$ we consider the standard basis elements $T_j^{(1)}=1^{(a+1)}j\otimes 2\cdots\hat{j} \cdots b$ of $Hom(D_{a+1}\otimes {D_1}^{\otimes (b-1)}, D_{a+2} \otimes \Lambda^{b-2})$ for $j\in\{2,\dots , b\}$.
Then we have the relations $\sum_{s=1}^{b-1} (-1)^{s-1}\pi(\theta_s (T_j^{(1)}))=0$, for every $j\in\{2,\dots ,b\}$, and so
\begin{equation} \label{e1}
\sum_{j=2}^b (-1)^{j-1} \pi(\theta_1 (T_j^{(1)})) + \sum_{s=2}^{b-1} (-1)^{s-1} \sum_{j=2}^b (-1)^{j-1}\pi(\theta_s (T_j^{(1)})) =0.
\end{equation}
But $\theta_2 (T_j^{(1)})=0$ for $j\geq 4$, $\theta_s (T_j^{(1)})=0$ for
$j \in \{2,\dots,s-1\}\cup \{s+1,\dots b\}$, and $\theta_{b-1}(T_j^{(1)})=0$ for $j\leq b-2$. So $(\ref{e1})$ is equivalent to 
\begin{equation} \label{e2}
\sum_{j=2}^b (-1)^{j-1} \pi(\theta_1 (T_j^{(1)})) + \sum_{s=2}^{b-1} (\pi(\theta_s (T_s^{(1)})) -\pi(\theta_s (T_{s+1}^{(1)}))) =0.
\end{equation}
We observe now that $\theta_s (T_s^{(1)}) = \theta_s (T_{s+1}^{(1)}) = 1^{(a+1)}s \otimes 2\cdots (b-1)$ and $\theta_1 (T_2^{(1)})= (a+2) b_1 ^{(1)}$ and $\theta_1(T_j^{(1)})=b_{j-1}^{(1)}$ for $j\in\{3,\dots , b\}$, so $(\ref{e2})$ is equivalent to 
\begin{equation*}
\sum_{j=2}^{b-1} (-1)^{j} \pi(b_j^{(1)}) = (a+2)\pi(b_1^{(1)}).
\end{equation*}
For $i=2$, we consider the standard basis elements $T_1^{(2)}=1^{(a)}2^{(2)}\otimes 3\dots b$ and $T_j^{(2)}=1^{(a-1)}2j\otimes 1\hat{2}\dots \hat{j}\dots b$ in $Hom(D_a \otimes D_2 \otimes D_1 \otimes \dots \otimes D_1,D_{a+2}\otimes \Lambda^{b-2})$, for $j\in\{3,\dots ,b\}$. Then $\sum_{s=1}^{b-1} {(-1)}^{s-1}\pi(\theta_s (T_1^{(2)}))=0$ and $\sum_{s=1}^{b-1}{(-1)}^{s-1} \pi(\theta_s (T_j^{(2)}))=0$ for $j\in\{3,\dots ,b\}$. So we obtain
\begin{equation*} \label{e3} \sum_{j=3}^{b}{(-1)}^{j-1}\sum_{s=1}^{b-1}{(-1)}^{s-1} \pi(\theta_s (T_j^{(2)}))=0,
\end{equation*}
which is equivalent to
\begin{align}\label{e7}
&\sum_{j=3}^{b} {(-1)}^{j-1} \pi(\theta_1 (T_j^{(2)})) - \sum_{j=3}^{b} {(-1)}^{j-1} \pi(\theta_2 (T_j^{(2)})) \nonumber \\
&+ \sum_{s=3}^{b-1}{(-1)}^{s-1} \sum_{j=3}^{b}{(-1)}^{j-1}\pi(\theta_s (T_j^{(2)}))=0.
\end{align}
But $\theta_s (T_3^{(2)})=0$ for every $s\in\{4,\dots,b-1\}$, $\theta_s(T_b^{(2)})=0$ for every $s\in\{3,\dots b-2\}$, and $\theta_s (T_j^{(2)})=0$ for $s\in \{3,\dots,j-2\}\cup \{j+1,\dots,b-1\}$ for every $j\in\{4, \dots, b-1\}$ and also we have $\theta_s(T_s^{(2)})=\theta_s (T_{s+1}^{(2)})=1^{(a-1)}2^{(2)}s\otimes 1\hat{2}\dots \hat{s}\dots (b-1)$. As a result, $(\ref{e7})$ is equivalent to
\begin{equation} \label{e9}
\sum_{j=3}^{b} {(-1)}^{j-1} \pi(\theta_1 (T_j^{(2)})) - \sum_{j=3}^{b} {(-1)}^{j-1} \pi(\theta_2 (T_j^{(2)})) =0.
\end{equation}
But $\theta_1(T_j^{(2)})= \tbinom{a+1}{2}b_{j-1}^{(1)}$, for $j=3,\dots,b$, and thus
\begin{align}\label{e10}
&\sum_{j=3}^{b} {(-1)}^{j-1} \pi(\theta_1 (T_j^{(2)}))=\sum_{j=3}^{b} {(-1)}^{j-1}\tbinom{a+1}{2}\pi(b_{j-1}^{(1)})
= \tbinom{a+1}{2}\sum_{j=2}^{b-1}{(-1)}^j \pi(b_j^{(1)}) \nonumber \\
&=\tbinom{a+1}{2}(a+2)\pi(b_1^{(1)})=a\tbinom{a+2}{2}\pi(b_1^{(1)}),
\end{align}
where we used part (a) of the Proposition 4.5 in the second equality. We observe that $\theta_2(T_3^{(2)})=3b_2^{(2)}$ and $\theta_2(T_j^{(2)})=b_{j-1}^{(2)}$ for $j\in\{4,\dots,b\}$, and so 
\begin{equation}\label{e11}
 \sum_{j=3}^{b} {(-1)}^{j-1} \pi(\theta_2 (T_j^{(2)})) = 3\pi(b_2^{(2)}) + \sum_{j=4}^{b} {(-1)}^{j-1} \pi(b_{j-1}^{(2)}) =3\pi(b_2^{(2)}) + \sum_{j=3}^{b-1} {(-1)}^{j} \pi(b_{j}^{(2)}).
\end{equation}
Combining $(\ref{e9})$, $(\ref{e10})$ and $(\ref{e11})$ we obtain 
\begin{equation}\label{e12}
a\tbinom{a+2}{2}\pi(b_1^{(1)})+\sum_{j=3}^{b-1} {(-1)}^{j-1} \pi(b_{j}^{(2)})=3\pi(b_2^{(2)}).
\end{equation}
But $\sum_{s=1}^{b-1} {(-1)}^{s-1}\pi(\theta_s (T_1^{(2)}))=0$ and $\theta_s(T_1^{(2)})=0$ for every $s\in\{3,\dots,b-1\}$, because of the $3\cdots b$ part of $T_1^{(2)}$ in $\Lambda^{b-2}$, which implies that $\pi(\theta_1(T_1^{(2)}))=\pi(\theta_2(T_1^{(2)}))$ or equivalently $\tbinom{a+2}{2}\pi(b_1^{(1)})=\pi(b_1^{(2)}).$ Equation $(\ref{e12})$ is now equivalent to 
\begin{equation*}
a\pi(b_1^{(2)})+\sum_{j=3}^{b-1} {(-1)}^{j-1} \pi(b_{j}^{(2)})=3\pi(b_2^{(2)}),
\end{equation*}
and we have proved the case $i=2$. 

The general case for $i\geq 2$ will be proved by induction on $i$. Suppose that for $i\geq 2$ we have 
\begin{equation}
\label{inductive}
 a(-1)^i \pi(b_1^{(i)})+ (-1)^{i}\sum_{j=2}^{i-1} (-1)^{j-1} \pi(b_j^{(i)})+ {(-1)}^i\sum_{j=i+1}^{b-1} (-1)^{j-1} \pi(b_j^{(i)}) = 3\pi(b_i^{(i)}).
\end{equation}
Using (\ref{inductive}), Lemma \ref{bl1} yields
\begin{align}\label{1q1q}
a&(-1)^{i+1} \pi(b_1^{(i+1)})+ (-1)^{i+1}\sum_{j=2}^{i} (-1)^{j-1} \pi(b_j^{(i+1)}) \nonumber \\
&=a(-1)^{i+1} \pi(b_1^{(i)})+ (-1)^{i+1}\sum_{j=2}^{i-1} (-1)^{j-1} \pi(b_j^{(i)}) +3\pi(b_i^{(i)}) \nonumber \\
&=-\big(a(-1)^{i} \pi(b_1^{(i)})+ (-1)^{i}\sum_{j=2}^{i-1} (-1)^{j-1} \pi(b_j^{(i)})\big)+3\pi(b_i^{(i)}) \nonumber \\
&={(-1)}^i \sum_{j=i+1}^{b-1}{(-1)}^{j-1}\pi(b_j^{(i)}).
\end{align}
By Lemma \ref{1q1q1} we have 
\begin{equation}\label{1q1q1q}
\sum_{j=i+2}^{b-1}{(-1)}^{j-i}\pi(b_j^{(i+1)})=3\pi(b_{i+1}^{(i+1)})+\sum_{j=i+1}^{b-1}{(-1)}^{j-i}\pi(b_j^{(i)}),
\end{equation}
\noindent and adding (\ref{1q1q}) and (\ref{1q1q1q}) we obtain 
\begin{equation*}
a(-1)^{i+1} \pi(b_1^{(i+1)})+ (-1)^{i+1}\sum_{j=2}^{i} (-1)^{j-1} \pi(b_j^{(i+1)})+\sum_{j=i+2}^{b-1}{(-1)}^{j-i}\pi(b_j^{(i+1)})=
\end{equation*}
\begin{equation*}
{(-1)}^i\sum_{j=i+1}^{b-1}{(-1)}^{j-1}\pi(b_j^{(i)})+3\pi(b_{i+1}^{(i+1)})+\sum_{j=i+1}^{b-1}{(-1)}^{j-i}\pi(b_j^{(i)})=3\pi(b_{i+1}^{(i+1)}). 
\end{equation*} \end{proof}
\subsection{The image of the generator $\pi(\Gamma_{a,b})$}
By straightforward calculations, one verifies that the composition $i_2 \circ \pi_3: D_{a+3}\otimes \Lambda^{b-3}\rightarrow \Delta(h(3))\rightarrow D_{a+2}\otimes\Lambda^{b-2}$ is the map $f:D_{a+3}\otimes \Lambda^{b-3}\xrightarrow{\triangle \otimes 1} D_{a+2}\otimes D_1 \otimes \Lambda^{b-3}\xrightarrow{1 \otimes m}D_{a+2}\otimes \Lambda^{b-2}$.

Let $\phi: Ext^2(\Delta(h),D_{a+3}\otimes\Lambda^{b-3})\rightarrow Ext^2(\Delta(h),D_{a+2}\otimes\Lambda^{b-2})$ be the map induced by $i_2\circ \pi_3: D_{a+3}\otimes \Lambda^{b-3}\rightarrow  D_{a+2}\otimes \Lambda^{b-2}$, i.e. $\phi$ is the composition $i_2^{(2)}\circ \pi_3^{(2)}$,
\begin{align*}
\phi:&Ext^2(\Delta(h),D_{a+3}\otimes\Lambda^{b-3})\xrightarrow{\pi_3^{(2)}}Ext^2(\Delta(h),\Delta(h(3)))\\
&\xrightarrow{i_2^{(2)}}Ext^2(\Delta(h),D_{a+2}\otimes\Lambda^{b-2}).
\end{align*}
In this Subsection we will compute the image $\phi(\pi(\Gamma_{a,b}))$ of the cyclic generator $\pi(\Gamma_{a,b})$ that was defined in Proposition 4.1. More precisely, and using the notation established at the beginning of Subsection 4.2, we will prove the following. \begin{prop}We have \[\phi(\pi(\Gamma_{a,b}))=(a+b)\pi(\gamma_{a,b}),\] where $\gamma_{a,b}=\tbinom{a+2}{3} b_1^{(1)}+\sum_{i=2}^{b-1}{(-1)}^{i-1}b_i^{(i)}$ and $\pi(\gamma_{a,b})$ is a cyclic generator of $Ext^2(\Delta(h),D_{a+2}\otimes \Lambda^{b-2})$. 
\end{prop}
 \begin{proof}We recall from \cite{MS}, Lemma 4.4, that $\pi(\gamma_{a,b})$ is a cyclic generator of the group $Ext^2(\Delta(h),D_{a+2}\otimes \Lambda^{b-2}).$ 

As $\phi$ is a linear map, we obtain 
\begin{align*}
\phi(\pi(\Gamma_{a,b}))=&\sum_{i=3}^{b-2}{(-1)}^{i-1}\Big(a\phi(\pi(B_{i,1}^{a,b})) -\sum_{j=2}^{i-1} {(-1)}^j \phi(\pi(B_{i,j}^{a,b}))+\sum_{j=i+1}^{b-1} {(-1)}^{j} \phi(\pi(B_{i,j}^{a,b}))  \Big)\\
&+\tbinom{a+2} {3}\sum_{j=2}^{b-1}{(-1)}^{j}\phi(\pi(B_{1,j}^{a,b})),
\end{align*}
where empty sums are equal to zero. Using the straightening law we have
\begin{align*}
\phi(\pi(B_{1,j}^{a,b}))&={(-1)}^j\cdot \pi(1^{(a+2)}\otimes 2\dots (b-1)) + \pi(1^{(a+1)}j \otimes 1\dots \hat{j} \dots (b-1))\\&={(-1)}^j \pi(b_1^{(1)})+\pi(b_j^{(1)}).
\end{align*}
This implies that 
\begin{align}\label{eqgen1}
&\tbinom{a+2}{3}\sum_{j=2}^{b-1}{(-1)}^{j}\phi(\pi(B_{1,j}^{a,b}))=\tbinom{a+2}{3}\sum_{j=2}^{b-1} \Big(\pi( b_1^{(1)})+ {(-1)}^j\pi( b_j^{(1)})\Big)\nonumber \\
&=\tbinom{a+2}{3}\Big((b-2)\pi(b_1^{(1)}) + \sum_{j=2}^{b-1} {(-1)}^j \pi(b_j^{(1)})\Big)=\tbinom{a+2}{3}(a+b)\pi(b_1^{(1)}),
\end{align}
as $ \sum_{j=2}^{b-1} {(-1)}^j \pi(b_j^{(1)})=(a+2)\pi(b_1^{(1)})$, by part (a) of Proposition \ref{l1}.
Using again the straightening law, we observe that
\begin{itemize} 
\item $\phi(\pi(B_{i,j}^{a,b}))={(-1)}^{j-1}\pi(b_i^{(i)})+{(-1)}^i \pi(b_j^{(i)})$ for $1\leq j<i\leq b-1$, and
\item $\phi(\pi(B_{i,j}^{a,b}))={(-1)}^j \pi(b_i^{(i)})+{(-1)}^{i-1} \pi(b_j^{(i)})$ for $2\leq i<j\leq b-1$.
\end{itemize}
Then, by Proposition \ref{l1} for $i\in\{2,\dots,b-1\}$ and (\ref{eqgen1}), we have
\begin{align*}
\phi (\pi(\Gamma_{a,b}))&=\sum_{i=2}^{b-1}{(-1)}^{i-1}\Big(a\phi(\pi(B_{i,1}^{a,b})) -\sum_{j=2}^{i-1} {(-1)}^j \phi(\pi(B_{i,j}^{a,b}))+\sum_{j=i+1}^{b-1} {(-1)}^{j} \phi(\pi(B_{i,j}^{a,b}))  \Big)\\
&\,\,\,\,\,\,+\tbinom{a+2}{3}(a+b)\pi(b_1^{(1)}) \\
&=\sum_{i=2}^{b-1}{(-1)}^{i-1}\Big((a+b-3)\pi(b_i^{(i)}) + a{(-1)}^i\pi(b_1^{(i)})+{(-1)}^i\sum_{j=2}^{i-1}{(-1)}^{j-1}\pi(b_j^{(i)})\\
&\,\,\,\,\,\,+{(-1)}^i\sum_{j=i+1}^{b-1}{(-1)}^{j-1}\pi(b_j^{(i)})\Big)+\tbinom{a+2}{3}(a+b)\pi(b_1^{(1)})\\
&=(a+b)\Big(\tbinom{a+2}{3} \pi(b_1^{(1)})+\sum_{i=2}^{b-1}{(-1)}^{i-1}\pi(b_i^{(i)})\Big)=(a+b)\pi(\gamma_{a,b}),
\end{align*}
\end{proof}

Using Proposition 4.6, we will now prove the case $k=4$ of Theorem 1.1.
\begin{prop}
We have	$Ext^2(\Delta(h),\Delta(h(4)))=\mathbb{Z}_{gcd(3,a+b)}. $
\end{prop}
\begin{proof}We noted at the beginning of Section 4.3 that $\phi$ factors as follows
	\begin{align*}
	\phi:&Ext^2(\Delta(h),D_{a+3}\otimes\Lambda^{b-3})\xrightarrow{\pi_3^{(2)}}Ext^2(\Delta(h),\Delta(h(3)))\\
	&\xrightarrow{i_2^{(2)}}Ext^2(\Delta(h),D_{a+2}\otimes\Lambda^{b-2}).
	\end{align*}	
	We also saw in Proposition 4.6 that $$\phi(\pi(\Gamma_{a,b}))=(a+b)\pi(\gamma_{a,b}),$$ where $\pi(\Gamma_{a,b})$ is a generator of $Ext^2(\Delta(h),D_{a+3}\otimes \Lambda^{b-3})=\mathbb{Z}_3$ and $\pi(\gamma_{a,b})$ is a generator of $Ext^2(\Delta(h),D_{a+2}\otimes \Lambda^{b-2})=\mathbb{Z}_3$. Hence, $\phi$ is an isomorphism if and only if $3\notdivides a+b$ if and only if $\ker \phi =0$.
Using (\ref{long2}) for $k=3$, we obtain the exact sequence
\begin{align*}
0&\longrightarrow Ext^2(\Delta(h),\Delta(h(4)))\xrightarrow{i_3^{(2)}}Ext^2(\Delta(h),D_{a+3}\otimes\Lambda^{b-3})\\
&\xrightarrow{\pi_3^{(2)}}Ext^2(\Delta(h),\Delta(h(3))),
\end{align*}
and so $Ext^2(\Delta(h),\Delta(h(4)))=\ker \pi_3^{(2)}$. As $i_2^{(2)}$ is an injection, we obtain $\ker\phi = \ker \pi_3^{(2)}$ and so we have $Ext^2(\Delta(h),\Delta(h(4)))=\ker \phi$. We have $\ker \phi =0$ if and only if $3\notdivides a+b$ and, also, $\ker \phi=\mathbb{Z}_3$ if and only if $3| a+b$. We conclude that $Ext^2(\Delta(h),\Delta(h(4)))=\mathbb{Z}_{gcd(3,a+b)}$.
\end{proof}

\subsection{Modular $Ext^1$}
Let $K$ be an infinite field of characteristic $p > 0$. The corollary below is a statement for $S_K(n,r)$-modules, where $S_K(n,r)$ is the Schur algebra for $GL_n(K)$ corresponding to homogeneous polynomial representations of $GL_n(K)$ of degree $r$ \cite{Gr}. By $\Delta_K(h)$ we denote the Weyl module of highest weight $h$ for $S_K(n,r)$. From \cite{Ma} it follows that $Ext^1(\Delta_K(h),\Delta_K(h(1)))$ is one dimensional if $p|a+b$ and zero otherwise. Using the universal coefficient theorem (\cite{AB}, Thm. 5.3), Theorem 3.5 by \cite{MS}, and Theorem 1.1 we have the following complete picture.

\begin{cor}
	Let $h=(a,1^b)$ and  $h(k)=(a+k,1^{b-k})$, where $ 1 \le k \le b$. Suppose $n \ge b+1$. Then $Ext^1(\Delta_K (h),\Delta_K (h(k)))$ is 1-dimensional in each of the following cases:
		\begin{itemize}
			\item $k=1$ and $p$ divides $a+b$,
			\item $k=2$ and $p$ divides $ 2(a+b)/{gcd(2,a+b)}^2$,
	 		\item $k=3$ and $p$ divides $ 6/{gcd(2,a+b+1)gcd(3,a+b)}$,
			\item $k=4$ and $p$ divides $ {2gcd(3,a+b)}/{gcd(2,a+b)}$,
			\item $k\geq 5$ and $p=2$, $a+b+k$ is odd.
		\end{itemize}
		In all other cases, $Ext^1(\Delta_K (h),\Delta_K (h(k)))=0.$
\end{cor}

\end{document}